\documentclass[11]{amsart}
\usepackage{amsfonts,tikz,amsmath,mathtools,amsthm,enumitem,mathrsfs,latexsym,amssymb,blkarray,afterpage,mathrsfs}
\usepackage[margin=1in, paperwidth=8.5in,paperheight=11in]{geometry}
\usepackage{mleftright}
\usepackage{mathtools}
\usepackage{xr-hyper}
\usepackage{hyperref}

\usepackage[leqno]{amsmath}

\makeatletter
\newcommand{\leqnomode}{\tagsleft@true\let\veqno\@@leqno}
\newcommand{\reqnomode}{\tagsleft@false\let\veqno\@@eqno}
\makeatother

\usepackage{graphicx}

\allowdisplaybreaks

\begin{document}

\begin{abstract}
	
	An algebra $\mathcal{A}$ of $n\times n$ complex matrices is said to be \textit{idempotent compressible} if $E\mathcal{A}E$ is an algebra for all idempotents $E\in\mathbb{M}_n(\mathbb{C})$. Analogously, $\mathcal{A}$ is said to be \textit{projection compressible} if $P\mathcal{A}P$ is an algebra for all orthogonal projections $P$ in $\mathbb{M}_n(\mathbb{C})$. In this paper we construct several examples of unital algebras that admit these properties. In addition, a complete classification of the unital idempotent compressible subalgebras of $\mathbb{M}_3(\mathbb{C})$ is obtained up to similarity and transposition. It is shown that in this setting, the two notions of compressibility agree: a unital subalgebra of $\mathbb{M}_3(\mathbb{C})$ is projection compressible if and only if it is idempotent compressible. Our findings are extended to algebras of arbitrary size in \cite{ZCramerCompressibility}.

\end{abstract}

\title{Matrix Algebras with a Certain Compression Property I}

\thanks{${}^1$ Research supported in part by NSERC (Canada)}

\author[Z. Cramer]{{Zachary Cramer${}^1$}}

\author[L.W. Marcoux]{{Laurent W. Marcoux${}^1$}}

\author[H. Radjavi]{{Heydar Radjavi}}

\newcommand{\Addresses}{{
  \bigskip
  \footnotesize

  Zachary~Cramer, \textsc{Faculty of Mathematics, University of Waterloo,
    Waterloo, Ontario, N2L 3G1}\par\nopagebreak
  \textit{E-mail address}: \texttt{zcramer@uwaterloo.ca}

  \medskip

  Laurent~W.~Marcoux, \textsc{Department of Pure Mathematics, University of Waterloo,
    Waterloo, Ontario, N2L 3G1}\par\nopagebreak
  \textit{E-mail address}: \texttt{lwmarcoux@uwaterloo.ca}

  \medskip

  Heydar~Radjavi, \textsc{Department of Pure Mathematics, University of Waterloo,
    Waterloo, Ontario, N2L 3G1}\par\nopagebreak
  \textit{E-mail address}: \texttt{hradjavi@uwaterloo.ca}

}}

\date{\today}
\subjclass[2010]{15A30, 46H20} 
\keywords{Compression, Projection Compressibility, Idempotent Compressibility, Algebraic Corners}
\maketitle



\theoremstyle{plain}
\newtheorem{thm}{Theorem}[subsection]
\newtheorem{lem}[thm]{Lemma}
\newtheorem{cor}[thm]{Corollary}
\newtheorem{prop}[thm]{Proposition}
\newtheorem{conj}[thm]{Conjecture}
\newtheorem{exmp}[thm]{Example}
\newtheorem{rmk}[thm]{Remark}

\theoremstyle{definition}
\newtheorem{defn}[thm]{Definition}
\newtheorem{exer}[thm]{Exercise}
\newtheorem*{notation}{Notation}

\makeatletter
\@addtoreset{thm}{section}
\makeatother

\makeatletter
\def\@seccntformat#1{\@ifundefined{#1@cntformat}%
    {\csname the#1\endcsname\quad}
    {\csname #1@cntformat\endcsname}}
\newcommand{\section@cntformat}{\S\thesection\quad}
\newcommand{\subsection@cntformat}{\S\thesubsection\quad}
\makeatletter

\newcommand{\cc}{\mathbb{C}}
\newcommand{\rr}{\mathbb{R}}
\newcommand{\dd}{\mathbb{D}}
\newcommand{\cldd}{\overline{\mathbb{D}}}
\newcommand{\epsi}{\varepsilon}
\newcommand{\nn}{\mathbb{N}}
\newcommand{\zz}{\mathbb{Z}}
\newcommand{\fy}{\varphi}
\newcommand{\sign}{\text{sign}}
\newcommand{\bfs}{\textbf{S}}
\newcommand{\triv}{\textbf{1}}
\newcommand{\bb}{\textbf{B}}
\newcommand{\alga}{\mathcal{A}}
\newcommand{\hilb}{\mathcal{H}}
\newcommand{\inv}{\mathrm{GL}}
\newcommand{\nil}{\mathrm{Nil}}
\newcommand{\qnil}{\mathrm{QNil}}
\newcommand{\bh}{\mathcal{B(H)}}
\newcommand{\qh}{\mathcal{Q(H)}}
\newcommand{\ol}{\overline}
\newcommand{\mc}{\mathcal}
\newcommand{\dist}{\mathrm{dist}}
\newcommand{\nor}{\mathrm{Nor}}
\newcommand{\mm}{\mathbb{M}}
\newcommand{\au}{\sim_{au}}
\newcommand{\sorb}{\mathcal{S}}
\newcommand{\alg}{\mathrm{Alg}}
\newcommand{\bqt}{\mathrm{BQT}}
\newcommand{\anti}{\mathrm{Anti}}
\newcommand{\rad}{Rad(\mc{A})}
\newcommand{\rank}{\mathrm{rank}}
\newcommand{\ran}{\mathrm{ran}}
\newcommand\scalemath[2]{\scalebox{#1}} 
\renewcommand{\qedsymbol}{$\blacksquare$}

\newcommand\restr[2]{{
  \left.\kern-\nulldelimiterspace 
  #1 
  \vphantom{\big|} 
  \right|_{#2} 
  }}

\section{Introduction}\label{intro}

	In this paper we examine the following question: \textit{Which unital subalgebras $\mc{A}$ of $\mm_n(\cc)$ have the property that $E\mc{A}E$ is an algebra for all idempotents $E\in\mm_n(\cc)$? }
	
	Since for every idempotent $E$ one may decompose $\mc{A}$ as an algebra of block $2\times 2$ matrices with respect to the (potentially non-orthogonal) direct sum decomposition $\cc^n=\ran(E)\dotplus\ker(E)$, this question may be restated as follows: \textit{Which unital subalgebras $\mc{A}$ of $\mm_n(\cc)$ have the property that with respect to every direct sum decomposition $\cc^n=\ran(E)\dotplus\ker(E)$, the compression of $\mc{A}$ to the $(1,1)$-corner is an algebra of linear maps acting on $\ran(E)$?} This condition will be known as the \textit{idempotent compression property}. If $\mc{A}$ is a subalgebra of $\mm_n(\cc)$ for which this property holds, we say that $\mc{A}$ is \textit{idempotent compressible.}
	
	An interesting variant on the above problem arises when considering only the \textit{orthogonal} direct sum decompositions of $\cc^n$. If $\mc{A}$ is a subalgebra of $\mm_n(\cc)$ such that $P\mc{A}P$ is an algebra for all orthogonal projections $P\in\mm_n(\cc)$, we shall say that $\mc{A}$ exhibits the \textit{projection compression property} or that $\mc{A}$ is \textit{projection compressible.} It is clear that any algebra possessing the idempotent compression property must also be projection compressible.
	
	If $E\in\mm_n(\cc)$ is an idempotent, then the corner $E\mc{A}E$ is always a linear space. This means that $E\mc{A}E$ is an algebra if and only if it is multiplicatively closed. It is easy to see that this holds trivially for any idempotent from the algebra $\mc{A}$ itself. Furthermore, dimension considerations imply that this is also true for any idempotent of rank $1$. It follows that any subalgebra of $\mm_2(\cc)$ is trivially idempotent compressible, and hence projection compressible as well. Our study will therefore only concern subalgebras of $\mm_n(\cc)$ for integers $n\geq 3$. 

	While it is immediate from the definitions that every idempotent compressible algebra is also projection compressible, the converse is much less clear. As will be shown in $\S2$ and $\S3$, all of our preliminary examples indicate either the presence of the idempotent compression property or the absence of the projection compression property, thus providing evidence to the affirmative. Despite this evidence, our attempts at obtaining an intrinsic proof of the equivalence of these notions have been unsuccessful. Instead, a systematic case-by-case analysis is used to investigate whether or not such an equivalence exists. Our analysis reveals that the techniques for studying the compression properties for subalgebras of $\mm_3(\cc)$ differed significantly from those used for subalgebras of $\mm_n(\cc)$ when $n\geq 4$. For this reason, our study has been divided into two parts. 
	
	We begin our examination in \S2 by introducing the notation and basic theory surrounding these notions of compressibility. In \S3, we develop a library of algebras that admit the idempotent compression property. As we shall see, the unital idempotent compressible  algebras constructed in this section account for all idempotent compressible algebras in $\mm_3(\cc)$ up to similarity and transposition. In order to show that this is the case, we will require certain results on the structure theory for matrix algebras outlined in $\S4$. We then devote $\S5$ to the classification of unital idempotent compressible subalgebras of $\mm_3(\cc)$, ultimately proving that in this setting, the notions of projection compressibility and idempotent compressibility coincide.
	
	In \cite{ZCramerCompressibility}, the sequel to this paper, we further examine the compression properties for unital subalgebras of $\mm_n(\cc)$ when $n\geq 4$. The main result, \cite[Theorem 7.1.1]{ZCramerCompressibility}, states that the two notions of compressibility agree in this setting as well. In fact, it is shown that up to similarity and transposition, the unital algebras admitting one (and hence both) of the compression properties are exactly those outlined in \S3 of this paper. 

\section{Preliminaries}

	In this section we introduce some of the preliminary results concerning algebras that admit the idempotent or projection compression properties. Our first task will be to establish the notation and terminology that will be used throughout.
	
	 Since we will only be concerned with algebras of $n\times n$ matrices over $\cc$, we will write $\mm_n$ in place of $\mm_n(\cc)$ from here on.

We begin by investigating some permanence results for algebras admitting the idempotent or projection compression property. These facts will be used extensively without mention. The first result in this vein states that if $\mc{A}$ admits one of the compression properties, then so too does $P\mc{A}P$ for every projection $P$.

\begin{prop}
Let $\mc{A}$ be a subalgebra of $\mm_n$ that admits the idempotent (resp. projection) compression property, and let $P$ be a projection in $\mm_n$. When restricted to an algebra of linear maps acting on $\ran(P)$, the algebra $P\mc{A}P$ is idempotent (resp. projection) compressible.

\end{prop}

\begin{proof}
Assume that $\mc{A}$ is idempotent compressible. Given an idempotent $E$ acting on $\ran(P)$, we have that $PE=EP=E$. Thus, $E(P\mc{A}P)E=E\mc{A}E$ is an algebra, as $\mc{A}$ is idempotent compressible.

An analogous argument may be used in the case that $\mc{A}$ is projection compressible.
\end{proof}
\smallskip

	Note that the set of idempotents in $\mm_n$ is closed under transposition and similarity, whereas the set of projections in $\mm_n$ is closed under transposition and unitary equivalence. This leads to our second permanence property for compressible algebras, Proposition~\ref{equivalent conditions for compressibility prop}. In order to simplify the statement of this result, as well as much of the exposition in the sections to come, we first introduce the following definitions.
	
	\begin{defn}\label{transpose similarity definition}
Let $\mc{A}$ and $\mc{B}$ be subsets of $\mm_n$. Define the \textit{transpose} of $\mc{A}$ to be the set $$\mc{A}^T\coloneqq\left\{A^T:A\in\mc{A}\right\}.\smallskip$$ If $\mc{A}$ or $\mc{A}^T$ is similar to $\mc{B}$, we say that $\mc{A}$ and $\mc{B}$ are \textit{transpose similar}. If $\mc{A}$ or $\mc{A}^T$ is unitarily equivalent to $\mc{B}$, we say that $\mc{A}$ and $\mc{B}$ are \textit{transpose equivalent}.

\end{defn}

It is easy to verify that transpose similarity and transpose equivalence are equivalence relations that generalize the notions of similarity and unitary equivalence, respectively.

The proof of the following result follows immediately from the comments preceding Definition~\ref{transpose similarity definition}.
	
		\begin{prop}\label{equivalent conditions for compressibility prop}
		Let $\mc{A}$ and $\mc{B}$ be subalgebras of $\mm_n$. 
		\begin{itemize}
			\item[(i)]If $\mc{A}$ and $\mc{B}$ are transpose similar, then $\mc{A}$ is idempotent compressible if and only if $\mc{B}$ is idempotent compressible.
			
			\item[(ii)]If $\mc{A}$ and $\mc{B}$ are transpose equivalent, then $\mc{A}$ is projection compressible if and only if $\mc{B}$ is projection compressible. \smallskip
			
		\end{itemize}
	\end{prop}
	

	
	\begin{defn}
	
		Given $A\in\mm_n$, define the \textit{anti-transpose} of $A$ to be the matrix
		$$A^{aT}\coloneqq JA^TJ,$$
		where $J=J^*$ is the unitary matrix whose $(i,j)$-entry is $\delta_{j,n-i+1}$. If $\mc{A}$ is a subset of $\mm_n$, then we will define the \textit{anti-transpose} of $\mc{A}$ to be the set 
		$$\mc{A}^{aT}\coloneqq J\mc{A}^TJ=\left\{A^{aT}:A\in\mc{A}\right\}.\smallskip$$
		
	\end{defn}
	
	While transposition has the effect of reflecting a matrix about its main diagonal, anti-transposition has the effect of reflecting a matrix about its \textit{anti-diagonal} (i.e., the diagonal from the $(n,1)$-entry to the $(1,n)$-entry). 
	
	Since an algebra $\mc{A}$ and its anti-transpose $\mc{A}^{aT}$ are easily seen to be transpose equivalent, we obtain the following useful consequence of Proposition~\ref{equivalent conditions for compressibility prop}.
	
	\begin{cor}\label{anti-transpose is compressible cor}
	If $\mc{A}$ is a subalgebra of $\mm_n$, then $\mc{A}$ is idempotent (resp. projection) compressible if and only if $\mc{A}^{aT}$ is idempotent (resp. projection) compressible.
	
	\end{cor}
	
			Next we show that if an algebra $\mc{A}$ admits the idempotent (resp. projection) compression property, then so too does its unitization $\mc{A}+\cc I$. A counterexample following the proof of Corollary~\ref{LR-algebras are idempotent compressible} demonstrates that the converse is false.

	\begin{prop}\label{unitization of compressible algebra is compressible prop}
	
		If $\mc{A}$ is an idempotent (resp. a projection) compressible subalgebra of $\mm_n$, then its unitization 
		$$\tilde{\mc{A}}\coloneqq \mc{A}+\cc I\smallskip$$
		is idempotent (resp. projection) compressible. 	
	\end{prop}
	
	\begin{proof}
		
		Assume that $\mc{A}$ is idempotent (resp. projection) compressible, and let $E$ be a idempotent (resp. projection) in $\mm_n$. Let $A,B\in\mc{A}$ and $\alpha,\beta\in\cc$, so that $A+\alpha I$ and $B+\beta I$ define elements of $\tilde{\mc{A}}$. Since $EAE\cdot EBE$ belongs to $E\mc{A}E$, we can write $EAE\cdot EBE=ECE$ for some $C\in\mc{A}$. As a result,  
		\begin{align*}E(A+\alpha I)E\cdot E(B+\beta I)E&=EAE\cdot EBE+\beta EAE+\alpha EBE+\alpha\beta E\\
		&=E((C+\beta A+\alpha B)+\alpha\beta I)E
		\end{align*}
	Since $(C+\beta A+\alpha B)+\alpha\beta I$ belongs to $\widetilde{A}$, we conclude that $E\widetilde{\mc{A}}E$ is an algebra.
	\end{proof}
	\smallskip
	
	The following proposition describes an obvious sufficient condition for an algebra to exhibit the projection or idempotent compression property, and will be useful in building our first class of examples.

		\begin{prop}\label{LR proposition}

		Let $n$ be a positive integer, and let $\mathcal{A}$ be a subalgebra of $\mm_n$. If $AEB\in\mathcal{A}$ for all $A,$ $B\in\mathcal{A}$, and all idempotents (resp. projections) $E\in\mm_n$, then $\mathcal{A}$ is idempotent (resp. projection) compressible.
	
	\end{prop}

	\begin{proof}
		Let $E$ be an idempotent (resp. a projection) in $\mm_n$. Given $A,B\in\mathcal{A}$, we have that $AEB\in\mathcal{A}$, and hence $$(EAE)(EBE)=E(AEB)E\in E\mc{A}E.\smallskip$$ 
		This demonstrates that $E\mc{A}E$ is multiplicatively closed, and therefore $E\mc{A}E$ is an algebra.
	\end{proof}
	\smallskip
	
	The sufficient condition for idempotent compressibility from Proposition~\ref{LR proposition} strongly resembles the multiplicative absorption property satisfied by ideals. In particular, this result implies that any (one- or two-sided) ideal of $\mm_n$ exhibits the idempotent compression property. It will be shown in Corollary~\ref{LR-algebras are idempotent compressible} that this property also holds for the intersection of one-sided ideals, or equivalently, the intersection of a single left ideal with a single right ideal. Thus, we make following definition.

	\begin{defn}\label{def of LR algs}
		If $\mc{A}$ is a subalgebra of $\mm_n$ given by an intersection of a left ideal and a right ideal in $\mm_n$, then $\mc{A}$ is said to be an $\mc{LR}$\textit{-algebra}.
		
	\end{defn}
	
It is straightforward to show that any algebra that is transpose similar to an $\mc{LR}$-algebra $\mc{A}$ is again an $\mc{LR}$-algebra. Indeed, if $\mc{A}=\mc{L}\cap \mc{R}$ for some left ideal $\mc{L}$ and right ideal $\mc{R}$ of $\mm_n$, then $\mc{R}^T$ is a left ideal, $\mc{L}^T$ is a right ideal, and $\mc{A}^T=\mc{R}^T\cap\mc{L}^T$. Hence, $\mc{A}^T$ is also an $\mc{LR}$-algebra. If $\mc{B}$ is transpose similar to $\mc{A}$, then by replacing $\mc{A}$ with $\mc{A}^T$ if necessary, we may assume that $$\mc{B}=S^{-1}\mc{A}S=\left(S^{-1}\mc{L}S\right)\cap\left(S^{-1}\mc{R}S\right)\smallskip$$ for some invertible $S\in\mm_n$. Since $S^{-1}\mc{L}S$ and $S^{-1}\mc{R}S$ are left and right ideals of $\mm_n$, respectively, $\mc{B}$ is again an $\mc{LR}$-algebra.

It is well-known that the one-sided ideals in $\mm_n$ can be described in terms of projections. In particular, each left ideal of $\mm_n$ has the form $\mm_n Q$ for some orthogonal projection $Q$, while each right ideal has the form $P\mm_n$ for some orthogonal projection $P$. More generally, we have the following classical ring-theoretic result concerning $\mm_n$-submodules of the $n\times p$ and $p\times n$ matrices (see \cite[Theorem 3.3]{Lam}). This result will be used in $\S5$ and invoked extensively throughout the classification in \cite{ZCramerCompressibility}.

\begin{thm}\label{structure of modules over Mn}

Let $n$ and $p$ be positive integers. 
	\begin{itemize}
		\item[(i)]If $\mc{S}\subseteq\mm_{n\times p}$ is a left $\mm_n$-module, then there is a projection $Q\in\mm_p$ such that $\mc{S}=\mm_{n\times p}Q$.
		\item[(ii)]If $\mc{S}\subseteq\mm_{p\times n}$ is a right $\mm_n$-module, then there is a projection $P\in\mm_p$ such that $\mc{S}=P\mm_{p\times n}$.
	\end{itemize}
	
\end{thm}

\smallskip

	\begin{cor}\label{form of an LR alg cor}
	
	A subalgebra $\mc{A}$ of $\mm_n$ is an $\mc{LR}$-algebra if and only if there are projections $P$ and $Q$ in $\mm_n$ such that $\mc{A}=P\mm_nQ.$
	
	\end{cor}

	The description of $\mc{LR}$-algebras presented in Corollary~\ref{form of an LR alg cor} allows one to quickly verify that these algebras admit the idempotent compression property.
	
	\begin{cor}\label{LR-algebras are idempotent compressible}
	
		Every $\mc{LR}$-algebra is idempotent compressible.	
	
	\end{cor}	
	
	\begin{proof}
	
		Let $\mc{A}$ be an $\mc{LR}$-algebra, so $\mc{A}=P\mm_n Q$ for some projections $P$ and $Q$ in $\mm_n$. If $E$ is an idempotent in $\mm_n$, then for any $A,B\in\mc{A}$,
		$$AEB=(PAQ)E(PBQ)=P(AQEPB)Q\in P\mm_n Q=\mc{A}.\smallskip$$
		Thus, $\mc{A}$ satisfies the assumptions of Proposition~\ref{LR proposition} in the case of idempotents. We conclude that $\mc{A}$ is idempotent compressible.
	\end{proof}
\smallskip

As a concrete example, any algebra generated by a rank-one operator is an $\mc{LR}$-algebra, and hence is idempotent compressible. A proof of this fact is given below. Note that for vectors $x,y\in\cc^n$, the notation $x\otimes y^*$ is used to denote the rank-rank-one operator $z\mapsto\langle z,y\rangle x$ acting on $\cc^n$.

	\begin{prop}\label{rank 1 operators generate LR algebras prop}
	If $R\in\mm_n$ is an operator of rank $1$, then $Alg(R)$---the algebra generated by $R$---is an $\mc{LR}$-algebra. Consequently, $Alg(R)$ is idempotent compressible.	
	
	\end{prop}
	
	\begin{proof}
		In light of the remarks following Definition~\ref{def of LR algs}, it suffices to prove that $Alg(R)$ is similar to an $\mc{LR}$-algebra.

		Let $\{e_1,e_2,\ldots, e_n\}$ denote the standard basis for $\cc^n$. As a rank-one operator, $R$ is either nilpotent or a scalar multiple of an idempotent. If $R$ is nilpotent, then $R$ is unitarily equivalent to $e_1\otimes e_2^*$. Consequently, $Alg(R)$ is unitarily equivalent to $\cc e_1\otimes e_2^*$. If instead $R$ is a multiple of an idempotent, then $R$ is similar to $\alpha e_1\otimes e_1^*$ for some non-zero $\alpha\in\cc$. Consequently, $Alg(R)$ is similar to $\cc e_1\otimes e_1^*$. In either case, $Alg(R)$ is an $\mc{LR}$-algebra.
	\end{proof}
	
	The fact that $\mc{LR}$-algebras admit the idempotent compression property gives us a means to disprove the converse to Proposition~\ref{unitization of compressible algebra is compressible prop}. We will exhibit a subalgebra of $\mm_3$ that is not projection compressible, but whose unitization is idempotent compressible. 
	
	Indeed, let $\{e_1,e_2,e_3\}$ denote the standard basis for $\cc^3$ and for each $i$, let $Q_i$ denote the orthogonal projection onto the span of $\left\{e_i\right\}$. Consider the algebra $\mc{A}=\cc(Q_1+Q_2)$. Note that the unitization of $\mc{A}$ is also the unitization of the $\mc{LR}$-algebra $\mc{B}\coloneqq \cc Q_3=Q_3\mm_3Q_3.$ By Corollary~\ref{LR-algebras are idempotent compressible} and Proposition~\ref{unitization of compressible algebra is compressible prop}, $\tilde{A}$ is idempotent compressible, a fortiori, projection compressible.
	
	To see that $\mc{A}$ is not projection compressible, consider the matrix $$P=\begin{bmatrix}
	\phantom{-}2 & -1 & -1\\
	-1 & \phantom{-}2 & -1\\
	-1 & -1 & \phantom{-}2
	\end{bmatrix},$$ and note that $\frac{1}{3}P$ is a projection in $\mm_3$. We claim that $\left(\frac{1}{3} P\right)\mc{A}\left(\frac{1}{3}P\right)$ is \textit{not} an algebra. Of course, since $(\frac{1}{3}P)\mc{A}(\frac{1}{3}P)$ is an algebra if and only if $P\mc{A}P$ is an algebra, it suffices to prove that $P\mc{A}P$ is not multiplicatively closed. 
	
	One may verify that every element $B=(b_{ij})$ in $P\mc{A}P$ satisfies the equation $b_{22}+5b_{23}=0$. With $B=e_1\otimes e_1^*+e_2\otimes e_2^*$,
	however, we have that 
	$$(PBP)^2=\begin{bmatrix}
	\phantom{-}42 & -39 & -3\\
	-39 & \phantom{-}42 & -3\\
	-3 & -3 & \phantom{-}6
	\end{bmatrix}.$$ This matrix clearly does not satisfy the above equation, and hence $(PBP)^2$ does not belong to $P\mc{A}P$. Thus, $P\mc{A}P$ is not an algebra, so $\mc{A}$ is not projection compressible.
	
	\begin{rmk}
	
	 \upshape{When determining whether or not a corner $E\mc{A}E$ is an algebra, it is often more computationally convenient to consider a multiple of the idempotent $E$ rather than $E$ itself. This simplification will frequently be used without mention.\\	}
	
	\end{rmk}

	\section{Examples of Idempotent Compressible Algebras}\label{Section: Examples}	

While $\mc{LR}$-algebras comprise a large collection of idempotent compressible algebras, they are not the only examples. The purpose of \S\ref{Section: Examples} is to expand our library of matrix algebras that admit the idempotent compression property. 

In $\S\ref{Subsection: Subalgebras of Mn, n>=3}$ we present three distinct families of idempotent compressible algebras that arise as subalgebras of $\mm_n$ for each $n\geq 3$. In $\S\ref{Subsection: Subalgebras of M3}$, we highlight three additional examples of idempotent compressible algebras that occur uniquely in the setting of $3\times 3$ matrices. The algebras presented in these sections lay the groundwork for the classification of compressible algebras in \cite{ZCramerCompressibility}.

Note that we do not explicitly prove that the examples considered throughout $\S3$ are in fact, algebras. However, each example will be defined as a subspace of $\mm_n$ written in terms of orthogonal projections $Q_1$, $Q_2$, and $Q_3$ that sum to $I$. Since such projections necessarily have pairwise orthogonal ranges \linebreak \cite[Proposition~4.19]{Douglas}, it will follow that these subspaces are multiplicatively closed, and hence algebras. For convenience, each example will presented as a collection of block $3\times 3$ matrices written with respect to the orthogonal direct sum decomposition $\cc^n=\ran(Q_1)\oplus\ran(Q_2)\oplus\ran(Q_3)$.

In the special case that one or more of the projections $Q_1,Q_2,Q_3$ has rank $1$, the following Proposition will be useful in verifying the idempotent compression property for the algebra $\mc{A}$.

	\begin{prop}\label{prop on rank 1 operators}

If $E\in\mm_n$ is an operator of rank $1$, then the linear space $\cc E$ is an algebra, and $E\mm_n E$ is contained in $\cc E$.\smallskip

\end{prop}

\begin{proof}

Let $E$ be as above. As a rank-one operator, $E$ is either nilpotent or a scalar multiple of an idempotent. Hence, $\cc E$ is closed under multiplication. Writing $E=x\otimes y^*$ for some vectors $x,y\in\cc^n$, we have that for any $A\in\mm_n$, 
$$EAE=(x\otimes y^*)A(x\otimes y^*)=\langle Ax,y\rangle (x\otimes y^*)=\langle Ax,y\rangle E\in\cc E.\smallskip$$
Thus, $E\mm_n E\subseteq \cc E$.
\end{proof} 
\smallskip

	
\subsection{Subalgebras of $\mm_n$, $n\geq 3$}\label{Subsection: Subalgebras of Mn, n>=3}

This section is devoted to the exposition of three families of idempotent compressible algebras that exist in $\mm_n$ for each $n\geq 3$. These families are described in Examples~\ref{a family of compressible algebras}, \ref{a family of compressible algebras 2}, and \ref{horrible family of idempotent compressible algebras}, respectively.

	\begin{exmp}\label{a family of compressible algebras}
		
		Let $n\geq 3$ be an integer. If $Q_1$, $Q_2$, and $Q_3$ are projections in $\mm_n$ which sum to $I$, then the algebra
		$$\begin{array}{l}
		\mathcal{A}\coloneqq \cc Q_1+(Q_1+Q_2)\mm_n(Q_2+Q_3)=\left\{\begin{bmatrix}
		\alpha I & M_{12} & M_{13}\\
		0 & M_{22} & M_{23}\\
		0 & 0 & 0
		\end{bmatrix}:\alpha\in\cc,M_{ij}\in Q_i\mm_nQ_j\right\}
		\end{array}$$ has the idempotent compression property.
	Consequently, its unitization 
	$$\begin{array}{l}
	\widetilde{\mathcal{A}}=\cc Q_1+\cc Q_3+(Q_1+Q_2)\mm_n(Q_2+Q_3)=\left\{\begin{bmatrix}
		\alpha I & M_{12} & M_{13}\\
		0 & M_{22} & M_{23}\\
		0 & 0 & \beta I
		\end{bmatrix}:\alpha,\beta\in\cc,M_{ij}\in Q_i\mm_nQ_j\right\}
		\end{array}$$ has the idempotent compression property as well.
	
	\end{exmp}
	
	\begin{proof}
		Define $\mathcal{A}_1\coloneqq \cc Q_1$ and $\mathcal{A}_2\coloneqq (Q_1+Q_2)\mm_n(Q_2+Q_3),$ so that $\mathcal{A}=\mathcal{A}_1\dotplus \mathcal{A}_2$. Let $E$ be an idempotent in $\mm_n$. We will show that $E\mc{A}E$ contains the product $E\mc{A}_iE\cdot E\mc{A}_jE$ for each choice of $i$ and $j$.
		
Since $\mc{A}_2$ is an $\mc{LR}$-algebra, it is easy to see that $(E\mathcal{A}_2E)^2$ is contained in $E\mc{A}E$. What's more, the equation $Q_1=(Q_1+Q_2)Q_1$ shows that $E\mathcal{A}_1E\cdot E\mathcal{A}_2E$ is contained in $E\mathcal{A}_2E$, and hence in $E\mc{A}E$.
To see that $\left(E\mc{A}_1E\right)^2$ is contained in $E\mc{A}E$, write
		$$(EQ_1E)^2=EQ_1E-E(Q_1+Q_2)Q_1E\cdot E(Q_2+Q_3)E.\smallskip$$
	Finally, if $T\in\mm_n$, then the equation
	\begin{align*}	E(Q_1+Q_2)T(Q_2+Q_3)E\cdot EQ_1E&=E(Q_1+Q_2)T(Q_2+Q_3)E\\[0.5ex]&\,\,\,\,\,\,\,-E(Q_1+Q_2)T(Q_2+Q_3)E\cdot E(Q_2+Q_3)E,
	\end{align*}
proves that $E\mathcal{A}_2E\cdot E\mathcal{A}_1E$ is contained in $E\mathcal{A}E.$
	\end{proof}
\smallskip
	
	\begin{rmk}
	\upshape{
		Let $\left\{e_1,e_2,e_3\right\}$ denote the standard basis for $\cc^3$. For each $i\in\{1,2,3\}$, let $Q_i$ denote the orthogonal projection of $\cc^3$ onto $\cc e_i$. By Example~\ref{a family of compressible algebras}, the algebra $$\begin{array}{l}\mc{A}=\cc Q_1+\cc Q_3+(Q_1+Q_2)\mm_n(Q_2+Q_3)=\left\{\begin{bmatrix}
		\alpha & x & y\\
		0 & \beta & z\\
		0 & 0 & \gamma
		\end{bmatrix}:\alpha,\beta,\gamma,x,y,z\in\cc\right\}\end{array}$$ of all $3\times 3$ upper triangular matrices is idempotent compressible.\\
	}
	\end{rmk}

	\begin{exmp}\label{a family of compressible algebras 2}

		Let $n\geq 3$ be an integer. If $Q_1$ and $Q_2$ are mutually orthogonal rank-one projections in $\mm_n$, and $Q_3=I-Q_1-Q_2$, then the algebra
		$$\begin{array}{l}
		\mathcal{A}\coloneqq \cc Q_1+\cc Q_2+(Q_1+Q_2)\mm_nQ_3=\left\{\begin{bmatrix}
		\alpha & 0 & M_{13}\\
		0 & \beta & M_{23}\\
		0 & 0 & 0
		\end{bmatrix}:\alpha,\beta\in\cc, M_{ij}\in Q_i\mm_nQ_j\right\}\end{array}$$ has the idempotent compression property.
		Consequently, its unitization  
		$$\begin{array}{l}\widetilde{\mathcal{A}}=\cc Q_1+\cc Q_2+\cc Q_3+(Q_1+Q_2)\mm_nQ_3=\left\{\begin{bmatrix}
		\alpha & 0 & M_{13}\\
		0 & \beta & M_{23}\\
		0 & 0 & \gamma I
		\end{bmatrix}:\alpha,\beta,\gamma\in\cc,M_{ij}\in Q_i\mm_nQ_j\right\}\end{array}$$
		has the idempotent compression property as well.
	
	\end{exmp}
	
	\begin{proof}
		Define $\mathcal{A}_1\coloneqq \cc Q_1,$ $\mathcal{A}_2\coloneqq \cc Q_2,$ and $ \mathcal{A}_3\coloneqq (Q_1+Q_2)\mm_nQ_3,$ so that $\mathcal{A}=\mc{A}_1\dotplus\mc{A}_2\dotplus\mc{A}_3$. Let $E$ be an idempotent in $\mm_n$. As in the previous proof, we will show that $E\mc{A}E$ contains the product $E\mc{A}_iE\cdot E\mc{A}_jE$ for all choices of $i$ and $j$.
		
		Note that $\mc{A}_1$, $\mc{A}_2$, and $\mc{A}_3$ are $\mc{LR}$-algebras, so $E\mc{A}E$ contains $(E\mc{A}_iE)^2$ for all $i$. Moreover, it is easy to see that $E\mathcal{A}_1E\cdot E\mathcal{A}_3E$ and $E\mathcal{A}_2E\cdot E\mathcal{A}_3E$ are contained in $E\mathcal{A}E$. 
		 From these inclusions it follows that $E\mathcal{A}_1E\cdot E\mathcal{A}_2E$ and $E\mc{A}_2E\cdot E\mc{A}_1E$ are contained in $E\mathcal{A}E$, as  
		$$\begin{array}{rl}EQ_1E\cdot EQ_2E=EQ_1E-EQ_1E\cdot EQ_1E-E(Q_1+Q_2)Q_1E\cdot EQ_3E, & \text{and}\vspace{0.3cm}\\
		EQ_2E\cdot EQ_1E=EQ_2E-EQ_2E\cdot EQ_2E-E(Q_1+Q_2)Q_2E\cdot EQ_3E.\end{array}\smallskip $$
		
		The proof will be complete upon showing that $E\mathcal{A}_3E\cdot E\mc{A}_1E$ and $E\mathcal{A}_3E\cdot E\mc{A}_2E$ are contained in $E\mathcal{A}E$. To demonstrate that this is the case, 
		observe that for any $T\in\mm_n$,
		$$\begin{array}{l}
		E(Q_1+Q_2)TQ_3E\cdot EQ_1E=EQ_1TQ_3E\cdot EQ_1E-EQ_2TQ_3E\cdot EQ_2+EQ_2T(I-Q_3E)Q_3E.
		\end{array}\smallskip$$
By Proposition~\ref{prop on rank 1 operators}, the first two summands on the right-hand side of this equation belong to $E\mc{A}_1E$ and  $E\mc{A}_2E$, respectively. Moreover, the final summand belongs to $E\mc{A}_3E$. Consequently, $E\mathcal{A}_3E\cdot E\mathcal{A}_1E$ is contained in $E\mathcal{A}E$. The inclusion $E\mc{A}_3E\cdot E\mc{A}_2E\subseteq E\mc{A}E$ can be deduced in a similar fashion.
	\end{proof} 
	\smallskip
	
	It was fairly routine to verify that the algebras presented in Examples \ref{a family of compressible algebras} and \ref{a family of compressible algebras 2} admit the idempotent compression property. Showing that this condition holds for the algebra $\mc{A}$ in our next example is not so straightforward. We will first present two lemmas that describe sufficient conditions for an arbitrary corner of this algebra to be an algebra itself. It will then be shown in Example \ref{horrible family of idempotent compressible algebras} that every such corner of $\mc{A}$ must satisfy one of these conditions. This will prove that this algebra is indeed idempotent compressible.\smallskip

	\begin{lem}\label{compressible whenever EQ_2E belongs to EAE}
	
		Let $n\geq 3$ be an integer, let $Q_1,Q_2\in\mm_n$ be mutually orthogonal rank-one projections, and define $Q_3\coloneqq I-Q_1-Q_2$. Consider the subalgebra $\mc{A}$ of $\mm_{n}$ given by  
		$$\begin{array}{l}\mathcal{A}\coloneqq \cc(Q_1+Q_2)+Q_1\mm_{n}Q_2+(Q_1+Q_2)\mm_{n}Q_3=\left\{\begin{bmatrix}
		\alpha & x & M_{13}\\
		0 & \alpha & M_{23}\\
		0 & 0 & 0
		\end{bmatrix}:\alpha,x\in\cc,M_{ij}\in Q_i\mm_nQ_j\right\}.\end{array}$$
If $E$ is an idempotent in $\mm_n$ and $ E\mathcal{A}E$ contains $EQ_2E$, then $E\mathcal{A}E$ is an algebra.
	
	\end{lem}
	
	\begin{proof}
	
	Let $E$ be a fixed idempotent in $\mm_n$ and suppose that $EQ_2E\in E\mathcal{A}E$. Define
	$$\begin{array}{l}
	\mathcal{A}_0\coloneqq\cc Q_1+\cc Q_2+Q_1\mm_nQ_2+(Q_1+Q_2)\mm_nQ_3\vspace{0.2cm}\\
	\phantom{\mc{A}_0}\,=\cc Q_1+(Q_1+Q_2)\mm_n(Q_2+Q_3),
	\end{array}\smallskip$$
	and note that $\mc{A}_0$ is idempotent compressible by Example \ref{a family of compressible algebras}. It then follow that
	$$\begin{array}{l}
	E\mathcal{A}E=\cc E(Q_1+Q_2)E+EQ_1\mm_nQ_2E+E(Q_1+Q_2)\mm_nQ_3E\vspace{0.2cm}\\
	\phantom{E\mathcal{A}E}=\cc EQ_1E+\cc EQ_2E+EQ_1\mm_nQ_2E+E(Q_1+Q_2)\mm_nQ_3E\vspace{0.2cm}\\
	\phantom{E\mathcal{A}E}=E\mathcal{A}_0E\end{array}
$$
is an algebra.
	\end{proof} 
\smallskip

	\begin{lem}\label{EQ_1=Q_1 lemma}
	
		Let $n\geq 3$ be an integer, let $Q_1,Q_2\in\mm_n$ be mutually orthogonal rank-one projections, and define $Q_3\coloneqq I-Q_1-Q_2$. Let $\mathcal{A}$ denote the subalgebra of $\mm_{n}$ given by 
		$$\begin{array}{l}\mathcal{A}\coloneqq \cc(Q_1+Q_2)+Q_1\mm_nQ_2+(Q_1+Q_2)\mm_nQ_3=\left\{\begin{bmatrix}
		\alpha & x & M_{13}\\
		0 & \alpha & M_{23}\\
		0 & 0 & 0
		\end{bmatrix}:\alpha,x\in\cc, M_{ij}\in Q_i\mm_nQ_j\right\}.\end{array}$$
		If $E$ is an idempotent in $\mm_n$ such that $EQ_1=Q_1$, then $E\mathcal{A}E$ is an algebra.	
	
	\end{lem}
	
	\begin{proof}

		Let $E$ be an idempotent in $\mm_n$ such that $EQ_1=Q_1$. Define $\mathcal{A}_1\coloneqq \cc (Q_1+Q_2),$  $\mathcal{A}_2\coloneqq Q_1\mm_nQ_2,$ and $\mathcal{A}_3\coloneqq (Q_1+Q_2)\mm_nQ_3,$ so that $\mathcal{A}=\mathcal{A}_1\dotplus\mathcal{A}_2\dotplus\mathcal{A}_3.$ As in the previous examples, we will show that $E\mc{A}E$ contains the product $E\mc{A}_iE\cdot E\mc{A}_jE$ for all $i$ and $j$.
		
		Since $\mc{A}_2$ and $\mc{A}_3$ are $\mc{LR}$-algebras, it is easy to see that $E\mc{A}E$ contains $(E\mc{A}_2E)^2$ and $(E\mc{A}_3E)^2$. Moreover, it is clear that $E\mathcal{A}_2E\cdot E\mathcal{A}_3E$ is contained in $E\mc{A}_3E$, and hence in $E\mathcal{A}E$. Observe that since the algebra $\mc{A}_0\coloneqq \mc{A}_1\dotplus\mc{A}_3$ was shown to be idempotent compressible in Example \ref{a family of compressible algebras}, we have that $E\mathcal{A}_1E\cdot E\mathcal{A}_3E$, $E\mathcal{A}_3E\cdot E\mathcal{A}_1E$, and $(E\mathcal{A}_1E)^2$ are contained in $E\mc{A}_0E\subseteq E\mc{A}E$. Proving these inclusions directly is also straightforward. 
		
		The equation $EQ_1=Q_1$ will now be used to obtain the remaining inclusions. We have that for all $S$ and $T$ in $\mm_n$,
		$$\begin{array}{rclr}E(Q_1+Q_2)SQ_3E\cdot EQ_1TQ_2E&=&0,\vspace{0.2cm}\\ 
		E(Q_1+Q_2)E\cdot EQ_1TQ_2E&=&EQ_1TQ_2E,\,\, \text{and}\vspace{0.2cm}\\
		EQ_1TQ_2E\cdot E(Q_1+Q_2)E&=&EQ_1(TQ_2E)Q_2E .\end{array}\smallskip$$
The right-hand side of each expression above is easily seen to belong to $E\mc{A}E$. As a result, $E\mc{A}E$ contains $E\mc{A}_3E\cdot E\mc{A}_2E$, $E\mc{A}_1E\cdot E\mc{A}_2E$, and $E\mc{A}_2E\cdot E\mc{A}_1E$, as claimed.
	\end{proof} 
	\smallskip
	
		\begin{exmp}\label{horrible family of idempotent compressible algebras}
	
		Let $n\geq 3$ be a positive integer, let $Q_1$ and $Q_2$ be mutually orthogonal rank-one projections in $\mm_{n}$, and define $Q_3\coloneqq I-Q_1-Q_2$. If $\mathcal{A}$ is the subalgebra of $\mm_n$ given by 
$$\begin{array}{l}\mathcal{A}\coloneqq \cc(Q_1+Q_2)+Q_1\mm_nQ_2+(Q_1+Q_2)\mm_nQ_3=\left\{\begin{bmatrix}
		\alpha & x & M_{13}\\
		0 & \alpha & M_{23}\\
		0 & 0 & 0
		\end{bmatrix}:\alpha,x\in\cc,M_{ij}\in Q_i\mm_nQ_j\right\},\end{array}$$
		then $\mathcal{A}$ is idempotent compressible. Consequently, its unitization
		$$\begin{array}{l}\tilde{\mathcal{A}}=\cc(Q_1+Q_2)+Q_1\mm_nQ_2+(Q_1+Q_2)\mm_{n}Q_3+\cc Q_3=\left\{\begin{bmatrix}
		\alpha & x & M_{13}\\
		0 & \alpha & M_{23}\\
		0 & 0 & \beta I
		\end{bmatrix}:\alpha,\beta,x\in\cc,M_{ij}\in Q_i\mm_nQ_j\right\}\end{array}$$
		is also idempotent compressible.
	
	\end{exmp}
	
	\begin{proof}
		
		In light of Lemmas \ref{compressible whenever EQ_2E belongs to EAE} and \ref{EQ_1=Q_1 lemma}, it suffices to prove that if $r\in\{2,3,\ldots, n-1\}$ and $E$ is an idempotent in $\mm_n$ of rank $r$, then either $EQ_2E\in E\mathcal{A}E$ or $EQ_1=Q_1$. 
		
		Fix such an integer $r$ and idempotent $E$. Let $\mc{B}=\left\{e_1,e_2,\ldots, e_n\right\}$ be an orthonormal basis for $\cc^n$ such that $e_1\in\ran(Q_1)$ and $e_2\in\ran(Q_2)$, and consider the projection
		$$P\coloneqq\sum_{i=1}^re_i\otimes e_i^*.$$
Since $\rank(P)=r$, there is an invertible matrix $S=(s_{ij})$ in $\mm_n$ such that $E=SPS^{-1}$.
		
		The product $EQ_2E$ belongs $E\mathcal{A}E$ if and only if there is an $A\in\mathcal{A}$ such that $$PS^{-1}(A-Q_2)SP=0.\smallskip$$ In showing this equality it suffices to exhibit an $A\in\mc{A}$ such that $(A-Q_2)SP=0$. To this end, observe that for any $A\in\mathcal{A}$, the operator $B\coloneqq A-Q_2$ admits the following matrix representation with respect to the basis $\mc{B}$: 		
		$$B=\begin{bmatrix}
  \alpha & w_2 & w_3 & \cdots &  w_n\\
  0 & \alpha-1 & v_3 & \cdots & v_n \\
  0 & 0 & 0 & \cdots & 0 \\
  \vdots & \vdots & \vdots & \ddots & \vdots \\
  0 & 0 & 0 & \cdots & 0 
  \end{bmatrix}.$$\smallskip
Since the last $n-2$ rows of $B$ and the last $n-r$ columns of $P$ are zero, the product $BSP$ is zero whenever $(BS)_{ij}=0$ for all $i\in\{1,2\}$ and $j\in\{1,2,\ldots, r\}$. That is, such a $B$ exists if there is a solution to the following non-homogeneous $2r\times 2(n-1)$ system of linear equations:
$$\begin{blockarray}{ccccccccccc}
w_2 & w_3 & \cdots & w_n & \alpha & v_3 & \cdots & v_n\\
\begin{block}{[cccccccc|cc]c}
  s_{21} & s_{31} & \cdots & s_{n1} & s_{11} & 0 & \cdots & 0 & 0\\
  s_{22} & s_{32} & \cdots & s_{n2} & s_{12} & 0 & \cdots & 0 & 0 \\
  \vdots & \vdots & \ddots & \vdots & \vdots & \vdots  & \ddots & \vdots & \vdots\\
  s_{2r} & s_{3r} & \cdots & s_{nr} & s_{1r} & 0  & \cdots & 0 & 0\\ 
  0 & 0 & \cdots & 0 & s_{21} & s_{31} &  \cdots & s_{n1} & s_{21}\\
  0 & 0 & \cdots & 0 & s_{22} & s_{32}  &  \cdots & s_{n2} & s_{22}\\
  \vdots & \vdots & \ddots & \vdots &  \vdots & \vdots & \ddots & \vdots & \vdots\\
  0 & 0 & \cdots & 0 & s_{2r} & s_{3r} & \cdots & s_{nr} & s_{2r}\vspace{0.1cm}\\ 
\end{block}
\end{blockarray}.$$
If the rank of the above (non-augmented) matrix is $2r$, then its columns span $\cc^{2r}$ and a solution exists. In this case, $EQ_2E$ belongs to $E\mc{A}E$, so $E\mathcal{A}E$ is an algebra by Lemma~\ref{compressible whenever EQ_2E belongs to EAE}.

Suppose that this is not the case, so the above (non-augmented) matrix has rank $<2r$. It is then apparent that 
$$S_0\coloneqq \begin{bmatrix}
s_{21} & s_{31} & \cdots & s_{n1}\\
s_{22} & s_{32} & \cdots & s_{n2}\\
\vdots & \vdots & \ddots & \vdots\\
s_{2r} & s_{3r} & \cdots & s_{nr}
\end{bmatrix}\smallskip$$
has rank $<r$. From here we will demonstrate that $EQ_1=Q_1$, or equivalently, that $PS^{-1}Q_1=S^{-1}Q_1$. 

To see that this is the case, note that if $S^{-1}=(t_{ij})$, then $t_{i1}=0$ for all $i>r$. Indeed,   
$$t_{i1}=\frac{C_{1i}}{\det(S)}\smallskip$$
 where $C_{ij}$ denotes the $(i,j)$-cofactor of $S$. When $i>r$, $C_{1i}$ is equal to $(-1)^{i+1}\det(M)$, where $M$ is an $(n-1)\times (n-1)$ matrix of the form 
$$M=\begin{bmatrix}
s_{21} & s_{22} & \cdots & s_{2r} & * & \cdots & *\\
s_{31} & s_{32} & \cdots & s_{3r} & * & \cdots & *\\
\vdots & \vdots & \ddots & \vdots & \vdots & \ddots & \vdots\\
s_{n1} & s_{n2} & \cdots & s_{nr} & * & \cdots & *
\end{bmatrix}.\smallskip$$
Since the $(n-1)\times r$ matrix obtained by keeping only the first $r$ columns of $M$ is exactly $S_0^T$ and $\mathrm{rank}(S_0)<r$, one has 
$$\mathrm{rank}(M)<r+(n-1-r)=n-1.\smallskip $$
Consequently, $t_{i1}=0$ for all $i>r$. A straightforward computation now shows that $PS^{-1}Q_1=S^{-1}Q_1$.
	\end{proof} 
\smallskip

	\subsection{Exceptional Subalgebras of $\mm_3$}\label{Subsection: Subalgebras of M3}
	
		In $\S\ref{Subsection: Subalgebras of Mn, n>=3}$ we introduced various examples of unital idempotent compressible subalgebras of $\mm_n$ for each integer $n\geq 3$. In \cite{ZCramerCompressibility} it will be shown that when $n\geq 4$, these examples are the only unital idempotent compressible subalgebras of $\mm_n$ up to similarity and transposition. In fact, we will see that for $n\geq 4$, our examples also represent all unital \textit{projection} compressible subalgebras of $\mm_n$ up to similarity and transposition.
		
		Unfortunately, the story for unital subalgebras of $\mm_3$ is somewhat more complicated. As we will see in this section, there exist several examples of unital idempotent compressible subalgebras of $\mm_3$ that are not accounted for in $\S\ref{Subsection: Subalgebras of Mn, n>=3}$. One explanation as to why these pathological examples arise is due to dimension. Just as $\mm_2$ is simply ``too small'' to contain the projections required to disprove the existence of the compression properties for any of its subalgebras, certain subalgebras of $\mm_3$ acquire the compression properties because $\mm_3$ does not contain projections of large enough rank. Support for this explanation is given by \cite[Theorem~2.0.5]{ZCramerCompressibility}, which demonstrates that in the case of $\mm_n$, $n\geq 4$, one can very often prove that an algebra lacks the compression properties using projections of rank $3$.\smallskip

			\begin{exmp}\label{T2 direct sum C is projection compressible example}
		
		Let $Q_1$, $Q_2$, and $Q_3$ be rank-one projections in $\mm_3$ that sum to $I$. If $\mathcal{A}$ is the subalgebra of $\mm_3$ defined by 
		$$\begin{array}{l}\mathcal{A}\coloneqq \cc Q_1+\cc Q_2+(Q_2+Q_3)\mm_3Q_3=\left\{\begin{bmatrix}
		\alpha & 0 & 0\\
		0 & \beta & x\\
		0 & 0 & \gamma
		\end{bmatrix}:\alpha,\beta,\gamma,x\in\cc\right\},\end{array}\smallskip$$
		then $\mathcal{A}$ is idempotent compressible.

	\end{exmp}

	\begin{proof}
	
		Define $\mathcal{A}_1\coloneqq \cc Q_1,$  $\mathcal{A}_2\coloneqq \cc Q_2,$ and $ \mathcal{A}_3\coloneqq (Q_2+Q_3)\mm_3 Q_3,$ so $\mathcal{A}=\mathcal{A}_1\dotplus \mathcal{A}_2\dotplus \mathcal{A}_3$. Let $E$ be an idempotent in $\mm_3$. We will show that $E\mc{A}E$ contains the product $E\mc{A}_iE\cdot E\mc{A}_j E$ for all $i$ and $j$. 
		
		For each $i\in\{1,2,3\}$, $\mc{A}_i$ is an $\mc{LR}$-algebra; hence $(E\mc{A}_iE)^2\subseteq E\mc{A}_iE\subseteq E\mc{A}E$. Moreover, since $EQ_2\mm_3Q_3E$ is contained in $E\mathcal{A}_3E$, we have that $E\mathcal{A}_2E\cdot E\mathcal{A}_3E\subseteq E\mathcal{A}E$ as well. These inclusions, together with the identities 
		\begin{align*}
		EQ_1E\cdot EQ_2E=EQ_1E-E&Q_1E\cdot EQ_1E-EQ_3E+EQ_2E\cdot EQ_3E+EQ_3E\cdot EQ_3E
		\end{align*}
		and 
		$$EQ_2E\cdot EQ_1E=EQ_2E-EQ_2E\cdot EQ_2E-EQ_2E\cdot EQ_3E,\smallskip$$
		
		\noindent demonstrate that $E\mathcal{A}_1E\cdot E\mathcal{A}_2E$ and $E\mathcal{A}_2E\cdot E\mathcal{A}_1E$ are contained in $E\mathcal{A}E$. Furthermore, if $T$ is an arbitrary element of $\mm_3$, then by writing  
		$$\begin{array}{l}
		EQ_1E\cdot E(Q_2+Q_3)TQ_3E=E(Q_2+Q_3)TQ_3E-E(Q_2+Q_3)E\cdot E(Q_2+Q_3)TQ_3E,\smallskip
		\end{array}$$
		it becomes apparent that $EQ_1E\cdot E(Q_2+Q_3)TQ_3E\in E\mc{A}E$. Consequently, $E\mathcal{A}_1E\cdot E\mathcal{A}_3E$ is contained in $E\mathcal{A}E$. 
		
		For the final inclusions, it will be helpful to first prove that $EQ_3E\cdot EQ_2E\in E\mathcal{A}E.$ Indeed, this is a consequence of the identity
		\begin{align*}
		EQ_3E\cdot EQ_2E=EQ_3E-E&Q_3E\cdot EQ_3E-EQ_1E+EQ_1E\cdot EQ_1E+EQ_2E\cdot EQ_1E
		\end{align*}
		and the inclusions established above. One may then apply Proposition~\ref{prop on rank 1 operators} to the rank-one operator $Q_3$ to deduce that $EQ_3\mm_3 Q_3E\cdot EQ_2E$ is contained in $E\mathcal{A}E$. Thus, for $T\in\mm_3$, we have that
		$$
		E(Q_2+Q_3)TQ_3E\cdot EQ_2E=EQ_2TQ_3E\cdot EQ_2E+EQ_3TQ_3E\cdot EQ_2E$$
	and
		\begin{align*}
		E(Q_2+Q_3)TQ_3E\cdot EQ_1E=E(Q_2+&Q_3)TQ_3E-E(Q_2+Q_3)TQ_3E\cdot EQ_3E\\
		&-EQ_2TQ_3E\cdot EQ_2E-EQ_3TQ_3E\cdot EQ_2E,
		\end{align*}
		
		 \noindent belong to $E\mc{A}E$. We conclude that $E\mc{A}E$ contains $E\mathcal{A}_3E\cdot E\mathcal{A}_2E$ and $E\mathcal{A}_3E\cdot E\mathcal{A}_1E$, and therefore $E\mathcal{A}E$ is an algebra.
	\end{proof} 
	\smallskip
	
	Proving the existence of the idempotent compression property for our next two examples will be somewhat more challenging. In the same spirit of the proof of Example \ref{horrible family of idempotent compressible algebras}, Examples \ref{second exceptional subalgebra example} and \ref{third exceptional subalgebra example} will each be preceded by two lemmas that highlight sufficient conditions for a corner of the algebra to be an algebra itself. We will then prove that all corners of these algebras must satisfy one of these two conditions.
	
	\begin{lem}\label{lemma 1 for exceptional M3 algebra 1}
	
		Let $Q_1$, $Q_2$, and $Q_3$ be rank-one projections in $\mm_3$ that sum to $I$. Let $\mc{A}$ be the subalgebra of $\mm_3$ defined by 
		$$\begin{array}{l}\mc{A}\coloneqq \cc(Q_1+Q_2)+\cc Q_3+Q_1\mm_3(Q_2+Q_3)=\left\{\begin{bmatrix}
		\alpha & x & y\\
		0 & \alpha & 0\\
		0 & 0 & \beta
		\end{bmatrix}:\alpha,\beta,x,y\in\cc\right\}.\end{array}$$
		If $E$ is an idempotent in $\mm_3$ such that $EQ_2E\in E\mc{A}E$, then $E\mc{A}E$ is an algebra.
		
	\end{lem}
	
	\begin{proof}
	
		Suppose that $E\in\mm_3$ is an idempotent such that $EQ_2E\in E\mc{A}E$, and define 
		$$\mc{A}_0\coloneqq \cc Q_1+\cc Q_2+\cc Q_3+Q_1\mm_3(Q_2+Q_3).\smallskip$$
		We have that 
		$$\begin{array}{l}
		E\mc{A}E=\cc E(Q_1+Q_2)E+\cc EQ_3E+EQ_1\mm_3(Q_2+Q_3)E\vspace{0.2cm}\\
\phantom{E\mc{A}E}=\cc EQ_1E+\cc EQ_2E+\cc EQ_3E+EQ_1\mm_3(Q_2+Q_3)E\vspace{0.2cm}\\
\phantom{E\mc{A}E}=E\mc{A}_0 E.
		\end{array}\smallskip$$
Since $\mc{A}_0^{aT}$ is the unital algebra from Example \ref{a family of compressible algebras 2}, $\mc{A}_0$ is idempotent compressible. Thus, $E\mc{A}_0 E=E\mc{A}E$ is an algebra.	
	\end{proof}
	\smallskip
	
		\begin{lem}\label{lemma 2 for exceptional M3 algebra 1}
	
		Let $Q_1$, $Q_2$, and $Q_3$ be rank-one projections in $\mm_3$ that sum to $I$. Let $\mc{A}$ be the subalgebra of $\mm_3$ defined by 
		$$\begin{array}{l}\mc{A}\coloneqq \cc(Q_1+Q_2)+\cc Q_3+Q_1\mm_3(Q_2+Q_3)=\left\{\begin{bmatrix}
		\alpha & x & y\\
		0 & \alpha & 0\\
		0 & 0 & \beta
		\end{bmatrix}:\alpha,\beta,x,y\in\cc\right\}.\end{array}\smallskip$$
		If $E$ is an idempotent in $\mm_3$ such that $EQ_1=Q_1$, then $E\mc{A}E$ is an algebra.
		
	\end{lem}
	
	\begin{proof}
	
		Let $E$ be an idempotent such that $EQ_1=Q_1$. Define $\mc{A}_1\coloneqq \cc (Q_1+Q_2)$, $\mc{A}_2\coloneqq \cc Q_3$, and $\mc{A}_3\coloneqq Q_1\mm_3(Q_2+Q_3),$ so that $\mc{A}=\mc{A}_1\dotplus\mc{A}_2\dotplus\mc{A}_3.$ To show that $E\mc{A}E$ is an algebra, we will verify that the product $E\mc{A}_iE\cdot E\mc{A}_jE$ is contained in $E\mc{A}E$ for all $i$ and $j$.
		
		Observe that $\mc{A}_2$ and $\mc{A}_3$ are $\mc{LR}$-algebras. Thus, $(E\mc{A}_iE)^2\subseteq E\mc{A}_iE\subseteq E\mc{A}E$ for each $i\in\{2,3\}$. Moreover, since \\
		$$\begin{array}{rcl}		
		E(Q_1+Q_2)E\cdot EQ_3E&=&EQ_3E-EQ_3E\cdot EQ_3E,\vspace{0.2cm} \\
		EQ_3E\cdot E(Q_1+Q_2)E&=&EQ_3E-EQ_3E\cdot EQ_3E,\,\,\,\,\,\,\text{and}\vspace{0.2cm}\\
		E(Q_1+Q_2)E\cdot E(Q_1+Q_2)E&=&E-2EQ_3E+EQ_3E\cdot EQ_3E,\,\,\,\,
		\end{array}\smallskip$$\\
		it follows that $E\mc{A}_1E\cdot E\mc{A}_2E$, $E\mc{A}_2E\cdot E\mc{A}_1E$, and $(E\mc{A}_1E)^2$ are all contained in $E\mc{A}E$. 
		
		For the remaining inclusions, note that for any $T\in\mm_3$,
		$$\begin{array}{l}
		EQ_1T(Q_2+Q_3)E\cdot E(Q_1+Q_2)E=EQ_1T(Q_2+Q_3)E \vspace{0.2cm}\\
		\left.\right.\hspace{6.7cm}-EQ_1T(Q_2+Q_3)E\cdot EQ_3(Q_2+Q_3)E
		\end{array}$$
		and 
		$$
		EQ_1T(Q_2+Q_3)E\cdot EQ_3E=EQ_1T(Q_2+Q_3)E\cdot EQ_3(Q_2+Q_3)E.\bigskip
		$$
		Consequently, $E\mc{A}_3E\cdot E\mc{A}_1E$ and $E\mc{A}_3E\cdot E\mc{A}_2E$ are contained in $E\mc{A}_3E\subseteq E\mc{A}E$. Finally, since $EQ_1=Q_1$ by hypothesis, we have that 
		$$\begin{array}{rrcll}\phantom{text{and}} & E(Q_1+Q_2)E\cdot EQ_1T(Q_2+Q_3)E&=&EQ_1T(Q_2+Q_3)E & \text{and}\vspace{0.2cm}\\
		& EQ_3E\cdot EQ_1T(Q_2+Q_3)E&=&0.\end{array}\smallskip$$
		This implies that $E\mc{A}E$ contains $E\mc{A}_1E\cdot E\mc{A}_3E$ and $E\mc{A}_2E\cdot E\mc{A}_3E$.
	\end{proof}
\smallskip

	\begin{exmp}\label{second exceptional subalgebra example}
	
		Let $Q_1$, $Q_2$, and $Q_3$ be rank-one projections in $\mm_3$ that sum to $I$. If $\mc{A}$ is the subalgebra of $\mm_3$ defined by 
		$$\begin{array}{l}\mc{A}\coloneqq \cc(Q_1+Q_2)+\cc Q_3+Q_1\mm_3(Q_2+Q_3)=\left\{\begin{bmatrix}
		\alpha & x & y\\
		0 & \alpha & 0\\
		0 & 0 & \beta
		\end{bmatrix}:\alpha,\beta,x,y\in\cc\right\},\end{array}\smallskip$$
		then $\mc{A}$ is idempotent compressible.
		
	\end{exmp}
	
	\begin{proof}
	
		It is obvious that $E\mc{A}E$ is an algebra whenever $E$ is an idempotent of rank $1$ or $3$. In light of Lemmas~\ref{lemma 1 for exceptional M3 algebra 1} and \ref{lemma 2 for exceptional M3 algebra 1}, it suffices to show that for every rank-two idempotent $E$ in $\mm_3$, either $EQ_2E$ belongs to $E\mc{A}E$ or $EQ_1=Q_1$. 
		
		To this end, suppose that $E$ is a rank-two idempotent in $\mm_3$ such that $EQ_2E$ does not belong to $E\mc{A}E$, and consider the projection $P\coloneqq (Q_1+Q_2)$. By rank considerations, there is an invertible matrix $S=(s_{ij})$ with inverse $S^{-1}=(t_{ij})$ such that $E=SPS^{-1}$. 
		
		Since $EQ_2E$ is not contained in $E\mc{A}E$, there is no $A\in\mc{A}$ that satisfies the equation $$SPS^{-1}(A-Q_2)SPS^{-1}=0.\smallskip$$ In particular, there is no $A\in\mc{A}$ such that $(A-Q_2)SP=0.$ Since every $A\in\mc{A}$ can be expressed as a matrix of the form 
		$$A=\begin{bmatrix}
		\alpha & x & y\\
		0 & \alpha & 0\\
		0 & 0 & \beta
		\end{bmatrix}\smallskip$$
		with respect to the decomposition $\cc^3=\ran(Q_1)\oplus\ran(Q_2)\oplus\ran(Q_3)$, it follows that there do not exist constants $\alpha, \beta, x, y\in\cc$ that solve the following system of linear equations:
		$$\left\{\begin{array}{ccl}
		\alpha s_{11}+xs_{21}+ys_{31}\phantom{+\beta s_{31}}&=&0\\
		\alpha s_{12}+xs_{22}+ys_{32}\phantom{+\beta s_{31}}&=&0\\
		\alpha s_{21}\phantom{+xs_{21}+ys_{31}+\beta s_{31}}&=&s_{21}\\
		\alpha s_{22}\phantom{+xs_{21}+ys_{31}+\beta s_{31}}&=&s_{22}\\
		\phantom{\alpha s_{22}+xs_{21}+ys_{31}+}\beta s_{31}&=&0\\
		\phantom{\alpha s_{22}+xs_{21}+ys_{31}+}\beta s_{32}&=&0
		\end{array}\right..\smallskip$$
		Note that if the determinant of 
		$S_0\coloneqq \begin{bmatrix} s_{21} & s_{31}\\ s_{22} & s_{32} \end{bmatrix}$ were non-zero, then a solution to the above system could be obtained by taking $\alpha=1$, $\beta=0$, and $x$ and $y$ such that $$x\begin{bmatrix} s_{21}\\ s_{22}\end{bmatrix}+y\begin{bmatrix} s_{31}\\ s_{32}
		\end{bmatrix}=\begin{bmatrix} -s_{11}\\ -s_{12}\end{bmatrix}.\smallskip$$ It must therefore be the case that $\det S_0=0.$
		
		We end the proof by showing that $EQ_1=Q_1$, or equivalently, that ${PS^{-1}Q_1=S^{-1}Q_1}$. It is easy to see that this equation holds when $t_{31}=0$. But if $C_{ij}$ denotes the $(i,j)$-cofactor of $S$, then indeed,
		$$t_{31}=\frac{C_{13}}{\det(S)}=\frac{\det(S_0^T)}{\det(S)}=0.$$
	\end{proof}
\smallskip

	\begin{lem}\label{lemma 1 for exceptional M3 algebra 2}
	
		Let $Q_1$, $Q_2$, and $Q_3$ be rank-one projections in $\mm_3$ that sum to $I$. Let $\mc{A}$ be the subalgebra of $\mm_3$ defined by 
		$$\mc{A}\coloneqq Q_1\mm_3(Q_2+Q_3)+Q_2\mm_3 Q_3+\cc I=\left\{\begin{bmatrix}
		\alpha & x & y\\
		0 & \alpha & z\\
		0 & 0 & \alpha
		\end{bmatrix}:\alpha,x,y,z\in\cc\right\}.$$
		If $E$ is an idempotent in $\mm_3$ such that $EQ_1E\in E\mc{A}E$, then $E\mc{A}E$ is an algebra.
		
	\end{lem}
	
	\begin{proof}
	
		Suppose that $E$ is an idempotent such that $EQ_1E\in E\mc{A}E$, and define
		$$\mc{A}_0\coloneqq \cc Q_1+\cc(Q_2+Q_3)+Q_1\mm_3(Q_2+Q_3)+Q_2\mm_3Q_3.\smallskip$$
		We have that 
		$$\begin{array}{l}
		E\mc{A}E=EQ_1\mm_3(Q_2+Q_3)E+EQ_2\mm_3Q_3E+\cc E\vspace{0.2cm}\\
		\phantom{E\mc{A}E}=EQ_1\mm_3(Q_2+Q_3)E+EQ_2\mm_3Q_3E+\cc EQ_1E+\cc E(Q_2+Q_3)E\vspace{0.2cm}\\
		\phantom{E\mc{A}E}=E\mc{A}_0 E.
		\end{array}\smallskip$$	
	Since $\mc{A}_0^{aT}$ is the unital algebra from Example \ref{horrible family of idempotent compressible algebras}, $\mc{A}_0$ is idempotent compressible. Thus, $E\mc{A}_0E=E\mc{A}E$ is an algebra.
	\end{proof}
	\smallskip
	
		\begin{lem}\label{lemma 2 for exceptional M3 algebra 2}
	
		Let $Q_1$, $Q_2$, and $Q_3$ be rank-one projections in $\mm_3$ that sum to $I$. 		Let $\mc{A}$ be the subalgebra of $\mm_3$ defined by 
		$$\mc{A}\coloneqq Q_1\mm_3(Q_2+Q_3)+Q_2\mm_3 Q_3+\cc I=\left\{\begin{bmatrix}
		\alpha & x & y\\
		0 & \alpha & z\\
		0 & 0 & \alpha
		\end{bmatrix}:\alpha,x,y,z\in\cc\right\}.$$
		If $E$ is an idempotent in $\mm_3$ such that $EQ_1=Q_1$, then $E\mc{A}E$ is an algebra.
		
	\end{lem}
	
	\begin{proof}
		
		Let $E$ be an idempotent such that $EQ_1=Q_1$. Define $\mc{A}_1\coloneqq Q_1\mm_3(Q_2+Q_3)$, $\mc{A}_2\coloneqq Q_2\mm_3Q_3$ and $\mc{A}_3\coloneqq \cc I,$ so that $\mc{A}=\mc{A}_1\dotplus\mc{A}_2\dotplus\mc{A}_3.$ Yet again, to show that $E\mc{A}E$ is an algebra, we will prove that the product $E\mc{A}_iE\cdot E\mc{A}_jE$ is contained in $E\mc{A}E$ for all $i$ and $j$.
		
		Observe that $E\mc{A}_iE\cdot E\mc{A}_jE$ is clearly contained in $E\mc{A}E$ when $i=3$ or $j=3$. Moreover, it is easy to see that $(E\mc{A}_1E)^2$ and $(E\mc{A}_2E)^2$ are contained in $E\mc{A}E$, as $\mc{A}_1$ and $\mc{A}_2$ are $\mc{LR}$-algebras. 
		
		Given $T,S\in\mm_3$, we have 
		$$EQ_1S(Q_2+Q_3)E\cdot EQ_2TQ_3E=EQ_1S(Q_2+Q_3)E\cdot EQ_2TQ_3(Q_2+Q_3)E,\smallskip$$
		so $E\mc{A}_1E\cdot E\mc{A}_2E$ is contained in $E\mc{A}_1E$, and hence in $E\mc{A}E$. Finally, we may use the fact that $EQ_1=Q_1$ to deduce that 
		$EQ_2SQ_3E\cdot EQ_1T(Q_2+Q_3)E=0,$
		and therefore $E\mc{A}_2E\cdot E\mc{A}_1E=\{0\}$.
	\end{proof}
	\smallskip
	
		\begin{exmp}\label{third exceptional subalgebra example}
	
		Let $Q_1$, $Q_2$, and $Q_3$ be rank-one projections in $\mm_3$ that sum to $I$. If $\mc{A}$ is the subalgebra of $\mm_3$ defined by 
		$$\mc{A}\coloneqq Q_1\mm_3(Q_2+Q_3)+Q_2\mm_3 Q_3+\cc I=\left\{\begin{bmatrix}
		\alpha & x & y\\
		0 & \alpha & z\\
		0 & 0 & \alpha
		\end{bmatrix}:\alpha,x,y,z\in\cc\right\},$$
		then $\mc{A}$ is idempotent compressible.
		
	\end{exmp}
	
	\begin{proof}
	
	It is obvious that $E\mc{A}E$ is an algebra whenever $E$ is an idempotent of rank $1$ or $3$. In light of Lemmas~\ref{lemma 1 for exceptional M3 algebra 2} and \ref{lemma 2 for exceptional M3 algebra 2}, it suffices to show that for every rank-two idempotent $E$ in $\mm_3$, either $EQ_1E$ belongs to $E\mc{A}E$, or $EQ_1=Q_1$. 
		
		To this end, suppose that $E$ is a rank-two idempotent in $\mm_3$ such that $EQ_1E$ does not belong to $E\mc{A}E$. Define $P\coloneqq (Q_1+Q_2)$, and let $S=(s_{ij})$ be an invertible matrix with inverse $S^{-1}=(t_{ij})$ satisfying $E=SPS^{-1}$.
		
		Since $EQ_1E$ is not contained in $E\mc{A}E$, then there is no $A\in\mc{A}$ satisfying the equation $$SPS^{-1}(A-Q_1)SPS^{-1}=0.\smallskip$$ In particular, there is no $A\in\mc{A}$ such that $(A-Q_1)SP=0.$ Since every $A\in\mc{A}$ can be expressed as a matrix of the form 
		$$A=\begin{bmatrix}
		\alpha & x & y\\
		0 & \alpha & z\\
		0 & 0 & \alpha
		\end{bmatrix}\smallskip$$
	with respect to the decomposition $\cc^3=\ran(Q_1)\oplus\ran(Q_2)\oplus\ran(Q_3)$, it follows that there do not exist constants $\alpha,x,y,z\in\cc$ that solve the following system of equations :
	$$\left\{\begin{array}{lcl}
	\alpha s_{11}+xs_{21}+ys_{31}\phantom{+zs_{21}}&=&s_{11}\\
	\alpha s_{12}+xs_{22}+ys_{32}\phantom{+zs_{21}}&=&s_{12}\\
	\alpha s_{21}\phantom{+xs_{22}+ys_{22}}
	\,\,\,\,+zs_{31}&=&0\\
	\alpha s_{22}\phantom{+xs_{22}+ys_{22}}\,\,\,\,+zs_{32}&=&0\\
	\alpha s_{31}\phantom{+xs_{22}+ys_{22}+zs_{32}}&=&0\\
	\alpha s_{32}\phantom{+xs_{22}+ys_{22}+zs_{32}}&=&0
	\end{array}\right..$$
			Observe, however, that if the determinant of
		$S_0\coloneqq \begin{bmatrix} s_{21} & s_{31}\\ s_{22} & s_{32} \end{bmatrix}$ were non-zero, then a solution could be obtained by taking $\alpha=z=0$, and $x$ and $y$ such that $$x\begin{bmatrix} s_{21}\\ s_{22}\end{bmatrix}+y\begin{bmatrix} s_{31}\\ s_{32}
		\end{bmatrix}=\begin{bmatrix} s_{11}\\ s_{12}\end{bmatrix}.\smallskip$$ It must therefore be the case that $\det S_0=0.$
		
		We are now prepared to show that $EQ_1=Q_1$, or equivalently, that $PS^{-1}Q_1=S^{-1}Q_1$. This equality is easily verified in the case that $t_{31}=0$. We have, however, that if $C_{ij}$ denotes the $(i,j)$ cofactor of $S$, then 
		$$t_{31}=\frac{C_{13}}{\det(S)}=\frac{\det(S_0^T)}{\det(S)}=0.$$
	\end{proof}
\smallskip
	
\section{Structure Theory for Matrix Algebras}\label{Section: Structure theory}

	In \S\ref{Section: Examples}, we introduced several families of unital algebras admitting the idempotent compression property. By Proposition~\ref{equivalent conditions for compressibility prop}, any algebra obtained by applying a transposition or similarity to one of these algebras also enjoys the idempotent compression property. It becomes interesting to ask whether or not this list is exhaustive. That is, is every unital idempotent compressible subalgebra of $\mm_n$ transpose similar to one of the idempotent compressible algebras we have encountered up to this point?  In order to decide whether or not additional examples exist, it will be necessary to establish a systematic approach to listing the unital subalgebras of $\mm_n$. Thus, this section will be devoted to recording a few key results concerning the structure theory for matrix algebras over $\cc$. The primary reference for this section is \cite{LMMRWedderburn}.
	
	Perhaps the most important result in this vein is the following theorem of Burnside \cite{Burnside},  which states that the only irreducible subalgebra of $\mm_n$ is the entire matrix algebra $\mm_n$ itself. See \cite{LomonosovRosenthal} for a simple proof.
\begin{thm}[Burnside's Theorem]

	If $\mathcal{A}$ is an irreducible algebra of linear transformations on a finite-dimensional vector space $\mathcal{V}$ over an algebraically closed field, then $\mathcal{A}$ is the algebra of all linear transformations on $\mathcal{V}$.

\end{thm}


As a consequence of Burnside's Theorem, every proper subalgebra $\mc{A}$ of $\mm_n$ can be block upper triangularized with respect to some orthonormal basis for $\cc^n$. Since the diagonal blocks in this decomposition are themselves algebras, Burnside's Theorem may be applied to these blocks successively to obtain a maximal block upper triangularization of $\mc{A}$. 

\begin{defn}\cite[Definition 9]{LMMRWedderburn}\label{definition of reduced block upper triangular}
A subalgebra $\mathcal{A}$ of $\mm_n$ is said to have a \textit{reduced block upper triangular form} with respect to a decomposition $\cc^n=\mathcal{V}_1\dotplus \mathcal{V}_2\dotplus\cdots\dotplus\mathcal{V}_m$ if 
\begin{itemize}
	\item[(i)]when expressed as a matrix, each $A$ in $\mathcal{A}$ has the form  
$$A=\begin{bmatrix}
A_{11} & A_{12} & A_{13} & \cdots & A_{1m}\\
0 & A_{22} & A_{23} & \cdots & A_{2m}\\
0 & 0 & A_{33} & \cdots & A_{3m}\\
\vdots & \vdots & \vdots & \ddots & \vdots\\
0 & 0 & 0 & \cdots & A_{mm}
\end{bmatrix} \smallskip$$
with respect to this decomposition, and\smallskip

	\item[(ii)]for each $i$, the algebra $\mathcal{A}_{ii}\coloneqq \{A_{ii}:A\in\mathcal{A}\}$ is irreducible. That is, either $\mathcal{A}_{ii}=\{0\}$ and $\dim\mathcal{V}_i=1$, or $\mathcal{A}_{ii}=\mm_{\dim\mathcal{V}_i}$.

\end{itemize}

\end{defn}

If $\mc{A}$ is a reduced block upper triangular algebra and $A\in\mc{A}$, define the \textit{block-diagonal} of $A$ to be the matrix $BD(A)$ obtained by replacing the block-`off-diagonal' entries of $A$ with zeros. In addition, define the \textit{block-diagonal} of $\mc{A}$ to be the algebra
	$$BD(\mathcal{A})=\left\{BD(A):A\in\mathcal{A}\right\}.\smallskip$$
 
	By definition, the non-zero diagonal blocks of a reduced block upper triangular matrix algebra $\mc{A}$ are full matrix algebras. There may, however, exist dependencies among different diagonal blocks. That is, while it may be the case that any matrix of suitable size can be realized as a diagonal block for some element of $\mc{A}$, there is no guarantee that matrices for different blocks can be chosen at will simultaneously.  The following result states that any dependencies that occur among the diagonal blocks of $\mc{A}$ can be described in terms of dimension and similarity.

\begin{thm}\textup{\cite[Corollary 14]{LMMRWedderburn}}\label{Cor 14 from LMMR}
If a subalgebra $\mathcal{A}$ of $\mm_n$ has a reduced block upper triangular form with respect to a decomposition $\cc^n=\mathcal{V}_1\dotplus\mathcal{V}_2\dotplus\cdots\dotplus\mathcal{V}_m$, then the set $\{1,2,\ldots, m\}$ can be partitioned into disjoint sets $\Gamma_1,\Gamma_2,\ldots, \Gamma_k$ such that   
\begin{itemize}
	\item[(i)] If $i\in\Gamma_s$ and $\mc{A}_{ii}\neq\{0\}$, then there exists $G^{<i>}\in\mathcal{A}$ such that $G_{jj}^{<i>}=I_{\mathcal{V}_j}$ for all $j\in\Gamma_s$, and $G_{jj}^{<i>}=0$ for all $j\notin\Gamma_s$.  \smallskip
	
	\item[(ii)]If $i$ and $j$ belong to the same $\Gamma_s$, then $\dim\mathcal{V}_i=\dim\mathcal{V}_j$, and there is an invertible linear map 
	$S_{ij}:\mathcal{V}_i\rightarrow\mathcal{V}_j$ such that  
	$$A_{ii}=S_{ij}^{-1}A_{jj}S_{ij}\,\,\text{for all}\,\,A\in\mc{A}.\smallskip$$

	\item[(iii)]If $i$ and $j$ do not belong to the same $\Gamma_s$, then  
	$$\left\{(A_{ii},A_{jj}):A\in\mathcal{A}\right\}=\left\{A_{ii}:A\in\mathcal{A}\right\}\times\left\{A_{jj}:A\in\mathcal{A}\right\}.\smallskip$$

\end{itemize}
\end{thm}

\begin{defn}

Let $\mathcal{A}$ be an algebra of the form described in Theorem~\ref{Cor 14 from LMMR}. Indices $i$ and $j$ are said to be \textit{linked} if they belong to the same $\Gamma_s$, and are said to be \textit{unlinked} otherwise. 

\end{defn}

It should be noted that if $\mc{A}$ is an algebra in reduced block upper triangular form and $S$ is an invertible matrix that is block upper triangular with respect to the same decomposition as that of $\mathcal{A}$, then $S^{-1}\mathcal{A}S$ has a reduced block upper triangular form with respect to this decomposition, and indices $i$ and $j$ are linked in $S^{-1}\mathcal{A}S$ if and only if they are linked in $\mathcal{A}$. 

The following Jordan-H\"{o}lder-type result describes the extent to which the reduced block upper triangular form of a subalgebra $\mc{A}$ of $\mm_n$ is unique.

\begin{thm}\textup{\cite[Theorem 23]{LMMRWedderburn}}
Suppose that a subalgebra $\mc{A}$ of $\mm_n$ has a reduced block upper triangular form with respect to a decomposition $\cc^n=\mc{V}_1\dotplus\mc{V}_2\dotplus\cdots\dotplus\mc{V}_k$, as well as with respect to a decomposition $\cc^n=\mc{W}_1\dotplus\mc{W}_2\dotplus\cdots\dotplus\mc{W}_m$. Then $k=m$ and there is a permutation $\pi$ on $\{1,2,\ldots, k\}$ such that 
\begin{itemize}
	\item[(i)]$i$ is linked to $j$ in the $\mc{V}$-decomposition if and only if $\pi(i)$ is linked to $\pi(j)$ in the $\mc{W}$-decomposition, and
	
	\item[(ii)]for each $i$ there is an invertible linear map $S_i:\mc{V}_i\rightarrow \mc{W}_{\pi(i)}$ such that $$\begin{array}{rl}
	\restr{A}{\mc{V}_i}=S_i^{-1}\restr{A}{\mc{W}_{\pi(i)}}S_i & \text{for all}\,\,A\in\mc{A}.\end{array}\smallskip$$
	
\end{itemize}

\end{thm}

The theorems presented above provide insight into the structure of the block-diagonal of a reduced block upper triangular matrix algebra $\mc{A}$. 
It will now be important to develop an understanding of the blocks that are located above the block-diagonal. 
%
%
%
%
%
%

Given a subalgebra $\mc{A}$ of $\mm_n$, it follows from \cite[Corollary 28]{LMMRWedderburn} that $\mc{A}$ decomposes as an algebraic direct sum $\mc{A}=\mc{S}\dotplus \rad$, where $\mc{S}$ is a semi-simple subalgebra of $\mc{A}$ and $\rad$ is the nil radical of $\mc{A}$. If $\mc{A}$ is in reduced block upper triangular form, then $\mc{S}$ is block upper triangular and $\rad$ consists of all strictly block upper triangular elements of $\mc{A}$ \cite[Proposition~19]{LMMRWedderburn}. Thus, the blocks above the block-diagonal are, in general, comprised of blocks from $\mc{S}$ and blocks from $\rad$. In the simplest scenario $\mc{S}$ is equal to $BD(\mc{A})$.

\begin{defn}
	Let $\mathcal{A}$ be a subalgebra of $\mm_n$ that has a reduced block upper triangular form with respect to some decomposition of $\cc^n$. 
	The algebra $\mathcal{A}$ is said to be \textit{unhinged} with respect to this decomposition if  
	$$\mathcal{A}=BD(\mathcal{A})\dotplus Rad (\mathcal{A}).\smallskip$$
\end{defn}

The following result indicates that if $\mc{A}$ is an algebra in reduced block upper triangular form with respect to some decomposition of $\cc^n$, then $\mc{A}$ can be unhinged with respect to this decomposition via conjugation by a block upper triangular similarity.

\begin{thm}\textup{\cite[Corollary 30]{LMMRWedderburn}}\label{every algebra is similar to an unhinged algebra} If a subalgebra $\mathcal{A}$ of $\mm_n$ has a reduced block upper triangular form with respect to a decomposition of $\cc^n$, then there exists an invertible linear operator $S$ that is block upper triangular with respect to the same decomposition as that of $\mc{A}$, and $S^{-1}\mathcal{A}S$ has an unhinged reduced block upper triangular form with respect to this decomposition.

\end{thm}

We remark that the transformation of an algebra $\mc{A}$ into an unhinged reduced block upper triangular form as described in Theorem~\ref{every algebra is similar to an unhinged algebra} can be achieved via application of a block upper triangular similarity, but not, in general, via unitary equivalence. We also note that in the special case that $\mc{A}$ is in reduced block upper triangular form and $BD(\mc{A})=\cc I$, Theorem~\ref{every algebra is similar to an unhinged algebra} implies that $\mc{A}=\cc I+\rad$. Thus, $\mc{A}$ is unhinged with respect to any decomposition in which it admits a reduced block upper triangular form.

We conclude this section with a lemma concerning the independence of the blocks in $\rad$, when $\mc{A}$ is an algebra that is in unhinged reduced block upper triangular form. This result will be used throughout \S5.\smallskip

\begin{lem}\label{middle block unlinked implies good radical lem}

	Let $n$ be a positive integer, and let $\mc{A}$ be a unital subalgebra of $\mm_n$ in reduced block upper triangular form with respect to a decomposition $\bigoplus_{i=1}^m \mc{V}_i$ of $\cc^n$. Suppose that there is an index $k$, $1<k<m$, that is unlinked from all indices $i\neq k$. Let $Q_1, Q_2,$ and $Q_3$ denote the orthogonal projections onto $\bigoplus_{i<k}\mc{V}_i$, $\mc{V}_k$, and $\bigoplus_{i>k}\mc{V}_i$, respectively, and assume that $Q_i\rad Q_i=\{0\}$ for all $i$. 
	
	If $\mc{A}$ is unhinged with respect to $\bigoplus_{i=1}^m\mc{V}_i$, then $$\rad=Q_1\rad Q_2\dotplus Q_1\rad Q_3\dotplus Q_2\rad Q_3.$$
%
%
%

\end{lem}

\begin{proof}
	Assume that $\mc{A}$ is unhinged with respect to the above decomposition, and let $R$ be an element of $\rad$. Since $\mc{V}_k$ is unlinked from $\mc{V}_i$ for all $i\neq k$, the projection $Q_2$ belongs to $\mc{A}$. Thus, the operators $RQ_2=Q_1 RQ_2$ and $Q_2R=Q_2RQ_3$ belong to $\rad$, and hence so too does $R-Q_2R-RQ_2=Q_1 RQ_3$. We conclude that $$\rad=Q_1\rad Q_2\dotplus Q_1\rad Q_3\dotplus Q_2\rad Q_3,$$ as claimed.
\end{proof}

		\section{Compressibility in $\mm_3$}

	We now turn our attention to assessing the completeness of the list of idempotent compressible algebras established in $\S\ref{Section: Examples}$. That is, we wish to determine whether or not there exist additional examples of unital idempotent compressible algebras up to transpose similarity.
	
	Our findings in $\S\ref{Subsection: Subalgebras of M3}$ suggest that there may exist pathological examples of such algebras in $\mm_3$. For this reason, we devote this section to classifying the unital subalgebras in $\mm_3$ that admit the idempotent compression property, and reserve the classification of such subalgebras of $\mm_n$, $n\geq 4$, for \cite{ZCramerCompressibility}.
	
	Using the structure theory established in  $\S\ref{Section: Structure theory}$, we will show in $\S\ref{Section: idempotent compressible subalgebras in M3}$ that up to transposition and similarity, the only unital idempotent compressible subalgebras of $\mm_3$ are those constructed in \S2 and \S\ref{Section: Examples}. As a consequence of this analysis, we will observe that a unital subalgebra $\mc{A}$ of $\mm_3$ that lacks the idempotent compression property is necessarily transpose similar to one of the following algebras:
		$$\begin{array}{rcll}
		\mc{B}&\coloneqq& \left\{\begin{bmatrix}
		\alpha & x & 0\\ 0& \alpha & 0\\ 0 & 0 & \beta
		\end{bmatrix}:\alpha,\beta,x\in\cc\right\},\vspace{0.2cm}\\
		\mc{C}&\coloneqq& \left\{\begin{bmatrix}
		\alpha & x & y\\ 0& \alpha & x\\ 0 & 0 & \alpha
		\end{bmatrix}:\alpha,x,y\in\cc\right\}, \,\,\,\text{or}\vspace{0.2cm}\\
		\mc{D}&\coloneqq& \left\{\begin{bmatrix}
		\alpha & 0 & 0\\ 0& \beta & 0\\ 0 & 0 & \gamma
		\end{bmatrix}:\alpha,\beta,\gamma\in\cc\right\}.
		\end{array}$$

	This observation has interesting implications for \textit{projection} compressibility $\mm_3$. In particular, it leads to an avenue for proving that in the case of unital subalgebras of $\mm_3$, the notions of projection compressibility and idempotent compressibility coincide. For if there were a unital projection compressible subalgebra $\mc{A}$ of $\mm_3$ that did not exhibit the idempotent compression property, then $\mc{A}$ must be similar to $\mc{B}$, $\mc{C}$, or $\mc{D}$. Thus, one could establish the equivalence of these notions by proving that no algebra similar to $\mc{B}$, $\mc{C}$, or $\mc{D}$ is projection compressible. This approach will be used in \S\ref{Section: Pcomp = Ecomp in M3}.

	\subsection{Classification of Idempotent Compressibility}\label{Section: idempotent compressible subalgebras in M3}	
		
	Here we begin the classification of unital idempotent compressible subalgebras of $\mm_3$ up to transposition and similarity. By the results outlined in $\S\ref{Section: Structure theory}$, we may assume that our algebras $\mc{A}$ are expressed in reduced block upper triangular form with respect to an orthogonal decomposition of $\cc^3=\bigoplus_{i=1}^m\mc{V}_i$, and are  unhinged with respect to this decomposition. That is, we will assume that
	$$\mc{A}=BD(\mc{A})\dotplus\rad,\smallskip$$ where $\rad$ consists of all strictly block upper triangular elements of $\mc{A}$. With this in mind, the algebras in this list will be organized according to the configuration of their block-diagonal and the dimension of their radical.
	
	If $\mc{A}=\mm_3$, then $\mc{A}$ is clearly idempotent compressible. Furthermore, if some $\mc{V}_i$ has dimension $2$, then Theorem~\ref{structure of modules over Mn} implies that $\mc{A}$ is transpose equivalent to  
	$\cc\oplus\mm_2$ or $$\left\{\begin{bmatrix}
	a_{11} & a_{12} & a_{13}\\
	0 & a_{22} & a_{23}\\
	0 & a_{32} & a_{33}
	\end{bmatrix}:a_{ij}\in\cc\right\}.\smallskip$$
	In either case, $\mc{A}$ is the unitization of an $\mc{LR}$-algebra, and hence is idempotent compressible. 
	
	Thus, we may assume from here on that all spaces $\mc{V}_i$ have dimension $1$. For each $i$, let $e_i$ be a unit vector in $\mc{V}_i$, and let $Q_i$ denote the orthogonal projection onto $\mc{V}_i$. \\

\noindent\textbf{Case I:} $\dim BD(\mc{A})=3$. If $\dim BD(\mc{A})=3$, then the spaces $\mc{V}_1$, $\mc{V}_2$, and $\mc{V}_3$ are mutually unlinked. An application of Lemma~\ref{middle block unlinked implies good radical lem} then shows that
	$$\rad=Q_1\rad Q_2\dotplus Q_1\rad Q_3\dotplus Q_2\rad Q_3.\smallskip$$

	\begin{itemize}
	\item[(i)] If $\rad=\{0\}$, then $\mc{A}$ is unitarily equivalent to $\mc{D}$, one of the three algebras presented at the outset of \S5. It will be shown in Theorem~\ref{no algebra similar to D is projection compressible} that no algebra similar to $\mc{D}$ is projection compressible. In particular, $\mc{A}$ is not idempotent compressible.\\
	
	\item[(ii)] If $\dim\rad=1$, then there is exactly one pair of indices $(i,j)$ such that $i<j$ and $Q_i\rad Q_j$ is non-zero. In this case, $\mc{A}$ is unitarily equivalent to $$\cc Q_1+\cc Q_2+(Q_2+Q_3)\mm_3 Q_3,$$ the algebra described in Example \ref{T2 direct sum C is projection compressible example}. Consequently, $\mc{A}$ is idempotent compressible.\\
	
	\item[(iii)] If $\dim\rad=2$, then $Q_i\rad Q_j=\{0\}$ for exactly one pair of indices $(i,j)$ with $i<j$. By considering products of elements in $\rad$, one can show that $Q_1\rad Q_3$ is non-zero whenever both $Q_1\rad Q_2$ and $Q_2\rad Q_3$ are non-zero. This means that either $Q_1\rad Q_2=\{0\}$ or $Q_2\rad Q_3=\{0\}$; hence $\mc{A}$ is transpose equivalent to $$\cc Q_1+\cc Q_2+\cc Q_3+(Q_1+Q_2)\mm_3Q_3.\smallskip$$ This algebra was shown to admit the idempotent compression property in Example~\ref{a family of compressible algebras 2}. Therefore, $\mc{A}$ is idempotent compressible. \\
	
	\item[(iv)] If $\dim\rad=3$, then $\mc{A}$ is equal to $$\cc Q_1+\cc Q_3+(Q_1+Q_2)\mm_3(Q_2+Q_3),\smallskip$$ the unital algebra from Example \ref{a family of compressible algebras}. Consequently, $\mc{A}$ is idempotent compressible.\\
	
	\end{itemize}
		
\noindent \textbf{Case II:} $\dim BD(\mc{A})=2$. 	If $\dim BD(\mc{A})=2$, then exactly two of the spaces $\mc{V}_i$ and $\mc{V}_j$ are linked. By replacing $\mc{A}$ with $\mc{A}^{aT}$ if necessary, we may assume that $\mc{V}_1$ is one of the linked spaces.
\begin{itemize}
	\item[(i)]If $\rad=\{0\}$, then $\mc{A}$ is unitarily equivalent to $\cc(Q_1+Q_2)+\cc Q_3$, the unitization of the $\mc{LR}$-algebra $\cc Q_3$. Consequently, $\mc{A}$ is idempotent compressible.\\
	
	\item[(ii)]If $\dim\rad=1$, then $\rad=\cc R$ for some strictly upper triangular element 
	$$R=\begin{bmatrix}
	0 & r_{12} & r_{13}\\
	0 & 0 & r_{23}\\
	0 & 0 & 0
	\end{bmatrix}.\smallskip$$
	Since $R^2\in\rad$, we have that $R^2=\alpha R$ for some $\alpha\in\cc$. From this it follows that at least one of $r_{12}$ or $r_{23}$ is equal to zero.
	
	First consider the case in which $\mc{V}_2$ is not linked to $\mc{V}_1$. By Lemma~\ref{middle block unlinked implies good radical lem},
	$$\rad=Q_1\rad Q_2\dotplus Q_1\rad Q_3\dotplus Q_2\rad Q_3.\smallskip$$
	If $r_{12}=r_{13}=0$ or $r_{13}=r_{23}=0$, then $\mc{A}$ or $\mc{A}^{aT}$ is equal to $$\mc{A}=Q_2\mm_3(Q_2+Q_3)+\cc I.\smallskip$$ In this case, $\mc{A}$ is idempotent compressible as it is the unitization of an $\mc{LR}$-algebra. If instead $r_{12}=r_{23}=0$, then $\mc{A}$ is unitarily equivalent to $\mc{B}$, one of the three algebras described at the beginning of \S5. It will be shown in Theorem~\ref{no algebra similar to B is projection compressible} that no algebra similar to $\mc{B}$ is projection compressible. In particular, $\mc{A}$ is not idempotent compressible.
	
	
	 Now consider the case in which $\mc{V}_1$ is linked to $\mc{V}_2$. Since $\mc{V}_3$ is therefore unlinked from $\mc{V}_1$ and $\mc{V}_2$, one may argue as in the proof of Lemma~\ref{middle block unlinked implies good radical lem} to show that $$\rad=Q_1\rad Q_2\dotplus (Q_1+Q_2)\rad Q_3.\smallskip$$ If $r_{12}=0$, then $\mc{A}$ is unitarily equivalent to $$(Q_2+Q_3)\mm_3Q_3+\cc I.$$ In this case, $\mc{A}$ is idempotent compressible as it is the unitization of an $\mc{LR}$-algebra. If instead $r_{12}\neq 0$, then $r_{13}=r_{23}=0$ and hence $\mc{A}$ is unitarily equivalent to $\mc{B}$.\\

	\item[(iii)]Suppose now that $\dim\rad=2$. If $\mc{V}_2$ is the unlinked space, then 
	$$\rad=Q_1\rad Q_2\dotplus Q_1\rad Q_3\dotplus Q_2\rad Q_3.\smallskip$$ It then follows that either $Q_1\rad Q_2=\{0\}$ or $Q_2\rad Q_3=\{0\}$, so $\mc{A}$ is transpose equivalent to
	$$\cc(Q_1+Q_2)+\cc Q_3+Q_1\mm_3(Q_2+Q_3).\smallskip$$
	This algebra was shown to admit the idempotent compression property in Example~\ref{second exceptional subalgebra example}, and hence $\mc{A}$ is idempotent compressible.
	
	Now consider the case in which $\mc{V}_2$ is linked to $\mc{V}_1$. Since $\mc{V}_3$ is therefore unlinked from $\mc{V}_1$ and $\mc{V}_2$, we have that
	$$\rad=Q_1\rad Q_2\dotplus(Q_1+Q_2)\rad Q_3.\smallskip$$
	If $Q_1\rad Q_2=\{0\}$, then $$\mc{A}=\mm_3Q_3+\cc I.\smallskip$$ Consequently, $\mc{A}$ is idempotent compressible as it is the unitization of an $\mc{LR}$-algebra. If instead $Q_1\rad Q_2=Q_1\mm_3Q_2$, then $(Q_1+Q_2)\rad Q_3$ is $1$-dimensional. Thus, there is a non-zero matrix $R\in(Q_1+Q_2)\mm_3Q_3$ such that 
	$$\rad=Q_1\mm_3 Q_2\dotplus \cc R.\smallskip$$
	It is then easy to see that $\langle Re_3,e_2\rangle=0.$ For if not, $\rad$ would contain an element of the form $e_2\otimes e_3^*+te_1\otimes e_3^*$ for some $t\in\cc$; hence
	$$e_1\otimes e_3^*=\left(e_1\otimes e_2^*\right)\left(e_2\otimes e_3^*+te_1\otimes e_3^*\right)\in\rad\smallskip$$ 
	This would then imply that $\rad$ is $3$-dimensional---a contradiction. 
	
	Thus, $\langle Re_3,e_2\rangle=0,$ so $\mc{A}$ is equal to $$\cc(Q_1+Q_2)+\cc Q_3+Q_1\mm_3(Q_2+Q_3),\smallskip$$ the idempotent compressible algebra from Example \ref{second exceptional subalgebra example}. In all cases, $\mc{A}$ is idempotent compressible.\\
	
	\item[(iv)]Suppose that $\dim\rad=3$. If $\mc{V}_2$ is the unlinked space, then $\mc{A}$ is equal to $$(Q_1+Q_2)\mm_3(Q_2+Q_3)+\cc I.\smallskip$$ In this case $\mc{A}$ is the unitization of an $\mc{LR}$-algebra, and hence is idempotent compressible. If instead $\mc{V}_2$ is linked to $\mc{V}_1$, then $\mc{A}$ is equal to $$\cc(Q_1+Q_2)+\cc Q_3+Q_1\mm_3Q_2+(Q_1+Q_2)\mm_3 Q_3,\smallskip$$ the unital algebra described in Example \ref{horrible family of idempotent compressible algebras}. Consequently, $\mc{A}$ is idempotent compressible.\\

\end{itemize}

\noindent \textbf{Case III:} $\dim BD(\mc{A})=1$. Suppose now that $\dim BD(\mc{A})=1$, so that all spaces $\mc{V}_i$ are mutually linked. That is, $BD(\mc{A})=\cc I$.
\begin{itemize}
	\item[(i)] If $\rad=\{0\}$, then $\mc{A}=\cc I$. Clearly $\mc{A}$ is idempotent compressible.\smallskip
	
	\item[(ii)] If $\dim\rad=1$, then $\rad=\cc R$ for some strictly upper triangular matrix
	$$R=\begin{bmatrix}
	0 & r_{12} & r_{13}\\
	0 & 0 & r_{23}\\
	0 & 0 & 0
	\end{bmatrix}.$$
	As in part (ii) of the previous case, one can show that $r_{12}=0$ or $r_{23}=0$, so $R$ necessarily has rank~$1$. By replacing $\mc{A}$ with $\mc{A}^{aT}$ if necessary, we may assume that $r_{12}=0$. But then $\mc{A}$ is unitarily equivalent to $$Q_2\mm_3Q_3+\cc I,\smallskip$$ the unitization of an $\mc{LR}$-algebra. Thus, $\mc{A}$ is idempotent compressible.\smallskip
	
	\item[(iii)] Suppose that $\dim\rad=2$. If $Q_1\rad Q_2=\{0\}$ or $Q_2\rad Q_3=\{0\}$, then $\mc{A}$ or $\mc{A}^{aT}$ is equal to $$Q_1\mm_3(Q_2+Q_3)+\cc I.\smallskip$$ Thus, $\mc{A}$ is idempotent compressible as it is the unitization of an $\mc{LR}$-algebra.
	
	Now consider the case in which $\rad$ contains an element 
	$$R=\begin{bmatrix}
	0 & r_{12} & r_{13}\\
	0 & 0 & r_{23}\\
	0 & 0 & 0
	\end{bmatrix}\smallskip$$
	with $r_{12}\neq 0$ and $r_{23}\neq 0$. When this occurs, $\rad$ contains $\frac{1}{r_{12}r_{23}}R^2=e_1\otimes e_3^*$; hence 
	$$\rad=\mathrm{span}\left\{e_1\otimes e_2^*+re_2\otimes e_3^*, e_1\otimes e_3^*\right\}\smallskip$$
	where $r\coloneqq r_{23}/r_{12}.$ Consequently, $$\mc{A}=\left\{\begin{bmatrix} \alpha & x & y\\
	0 & \alpha & rx\\
	0 & 0 & \alpha\end{bmatrix}:\alpha,x,y\in\cc\right\},$$
	which is easily seen to be similar to the algebra $\mc{C}$ described at the outset of \S5. It will be shown in Theorem~\ref{no algebra similar to Cr is projection compressible theorem} that no algebra similar to $\mc{C}$ is projection compressible. In particular, $\mc{A}$ is not idempotent compressible.\smallskip
	
	\item[(iv)] If $\dim\rad=3$, then $\mc{A}$ is equal to 
	$$Q_1\mm_3(Q_2+Q_3)+Q_2\mm_3Q_3+\cc I,\smallskip$$
	the idempotent compressible algebra described in Example \ref{third exceptional subalgebra example}.\\

\end{itemize}

This concludes our classification of the unital idempotent compressible subalgebras of $\mm_3$. Our findings are summarized in the following theorem.

\begin{thm}\label{summarizing thm on compressibility in M3}
Let $\mc{A}$ be a unital subalgebra of $\mm_3$.
\begin{itemize}
	\item[(i)] $\mc{A}$ is idempotent compressible if and only if $\mc{A}$ is the unitization of an $\mc{LR}$-algebra or transpose similar to one of the following algebras:	\reqnomode
		\begin{align}
		&\left\{\begin{bmatrix}
		\alpha & x & y\\ 0& \beta & z\\ 0 & 0 & \gamma
		\end{bmatrix}:\alpha,\beta,\gamma,x,y,z\in\cc\right\}, \tag{Example~\ref{a family of compressible algebras}}\\
		&\left\{\begin{bmatrix}
		\alpha & 0 & x\\ 0& \beta & y\\ 0 & 0 & \gamma
		\end{bmatrix}:\alpha,\beta,\gamma,x,y\in\cc\right\},\phantom{\gamma,z}\tag{Example~\ref{a family of compressible algebras 2}}\\
		&\left\{\begin{bmatrix}
		\alpha & x & y\\ 0& \alpha & z\\ 0 & 0 & \beta
		\end{bmatrix}:\alpha,\beta,x,y,z\in\cc\right\},\tag{Example~\ref{horrible family of idempotent compressible algebras}}\\
		&\left\{\begin{bmatrix}
		\alpha & 0 & 0\\ 0& \beta & x\\ 0 & 0 & \gamma
		\end{bmatrix}:\alpha,\beta,\gamma,x\in\cc\right\},\tag{Example~\ref{T2 direct sum C is projection compressible example}}\\
		&\left\{\begin{bmatrix}
		\alpha & x & y\\ 0& \alpha & 0\\ 0 & 0 & \beta
		\end{bmatrix}:\alpha,\beta,x,y\in\cc\right\},\tag{Example~\ref{second exceptional subalgebra example}}\\
		&\left\{\begin{bmatrix}
		\alpha & x & y\\ 0& \alpha & z\\ 0 & 0 & \alpha
		\end{bmatrix}:\alpha,x,y,z\in\cc\right\};\tag{Example~\ref{third exceptional subalgebra example}}
		\end{align}
	

	\item[(ii)]$\mc{A}$ fails to admit the idempotent compression property if and only if $\mc{A}$ is transpose similar to one of the algebras $\mc{B}$, $\mc{C}$, or $\mc{D}$, presented at the outset of \S5.\smallskip
\end{itemize}
\end{thm}

Although the algebras presented in Theorem~\ref{summarizing thm on compressibility in M3}(i) may appear to share little in common beyond the idempotent compression property, there do exist other interesting characterizations of this collection. For instance, aside from the unitizations of $\mc{LR}$-algebras, the unital idempotent compressible algebras are exactly those that are not $3$-dimensional. This equivalence was observed by Ken Davidson.\smallskip

\begin{cor}\label{3-dimensional corollary}

	A unital subalgebra of $\mc{A}$ of $\mm_3$ is idempotent compressible if and only if $\mc{A}$ is the unitization of an $\mc{LR}$-algebra, or $\mc{A}$ is not $3$-dimensional. \smallskip

\end{cor}

In addition, one may observe that the unital algebras lacking the idempotent compression property are exactly those that are generated by a matrix in which every Jordan block corresponds to a distinct eigenvalue. Such matrices are said to be \textit{nonderogatory} \cite[Definition~1.4.4]{HornJohnson}. This equivalence was observed by Rajesh Pereira. \smallskip

\begin{cor}\label{non-derog corollary}

	A unital subalgebra of $\mm_3$ is idempotent compressible if and only if it is not singly generated by an invertible nonderogatory matrix.\smallskip

\end{cor}

As we will see in \S5.2, the algebras described in Theorem~\ref{summarizing thm on compressibility in M3}(i) also represent the complete list of unital \textit{projection} compressible subalgebras of $\mm_3$ up to transpose similarity.

\subsection{Equivalence of Idempotent and Projection Compressibility}\label{Section: Pcomp = Ecomp in M3}

Our final goal of this section is to show that no unital subalgebra of $\mm_3$ can possess the projection compression property without also possessing the idempotent compression property. If such an algebra did exist, it would necessarily be transpose similar to $\mc{B}$, $\mc{C}$, or $\mc{D}$ by the analysis in $\S\ref{Section: idempotent compressible subalgebras in M3}$. Thus, to show that the notions of projection compressibility and idempotent compressibility agree for unital subalgebras of $\mm_3$, it suffices to prove that no algebra similar to $\mc{B}$, $\mc{C}$, or $\mc{D}$ is projection compressible. This goal will be accomplished by first characterizing the algebras similar to $\mc{B}$, $\mc{C}$, or $\mc{D}$ up to unitary equivalence. 

\begin{lem}\label{unitary orbit for bad algebra B}

	Let $\mc{A}$ be a subalgebra of $\mm_3$. If $\mc{A}$ is similar to 
	$$\mc{B}=\left\{\begin{bmatrix}\alpha & x & 0\\
	0 & \alpha & 0\\
	0 & 0 & \beta\end{bmatrix}:\alpha,\beta,x\in\cc\right\},$$
	then there are constants $s,t\in\cc$ such that $\mc{A}$ is unitarily equivalent to 
	$$\mc{B}_{st}\coloneqq \left\{\begin{bmatrix}
	\alpha & s(\alpha-\beta) & x\\
	0 & \beta & t(\alpha-\beta)\\
	0 & 0 & \alpha
	\end{bmatrix}:\alpha,\beta,x\in\cc\right\}.$$
	
	\end{lem}
	
	\begin{proof}
	
		If the matrices in $\mc{B}$ are expressed with respect to the standard basis $\{e_1,e_2,e_3\}$ for $\cc^3$, then $\mc{B}$ is spanned by $\left\{E_{11}+E_{22},E_{12},E_{33}\right\}$, where $E_{ij}\coloneqq e_i\otimes e_j^*.$ Thus, if $S$ is an invertible matrix in $\mm_3$ such that $\mc{A}=S^{-1}\mc{B}S$, then $\mc{A}$ is spanned by $\left\{E_{11}^\prime+E_{22}^\prime, E_{12}^\prime, E_{33}^\prime\right\},$ where $E_{ij}^\prime\coloneqq S^{-1}E_{ij}S.$ 
		
		Since $E_{12}^\prime$ is a rank-one nilpotent, there is a unitary $U\in\mm_3$ and a non-zero $y_0\in\cc$ such that 
		$$U^*E_{12}^\prime U=\begin{bmatrix}
		0 & 0 & y_0\\
		0 & 0 & 0\\
		0 & 0 & 0\end{bmatrix}.$$ Let $x_{ij}\in\cc$ be such that $U^*(E_{11}^\prime+E_{22}^\prime)U=(x_{ij})$. Using the fact that $$(E_{11}^\prime+E_{22}^\prime)E_{12}^\prime=E_{12}^\prime(E_{11}^\prime+ E_{22}^\prime)=E_{12}^\prime,\smallskip$$ one can show that $x_{21}=x_{31}=x_{32}=0$ and $x_{11}=x_{33}=1$. 
		Moreover, since $U^*(E_{11}^\prime+E_{22}^\prime)U$ is an idempotent
		of trace $2$,
		it follows that $x_{22}=0$ and $x_{13}=-x_{12}x_{23}$. Thus,
		$$U^*(E_{11}^\prime+E_{22}^\prime)U=\begin{bmatrix}
		1 & x_{12} & -x_{12}x_{23}\\
		0 & 0 & x_{23}\\
		0 & 0 & 1
		\end{bmatrix}.$$ Finally, we have that 
		$$U^*E_{33}^\prime U=I-U^*(E_{11}^\prime+E_{22}^\prime)U=\begin{bmatrix}
		0 & -x_{12} & x_{12}x_{23}\\
		0 & 1 & -x_{23}\\
		0 & 0 & 0
		\end{bmatrix}.$$
		As a result,  
		$$U^*\mc{A}U=\mathrm{span}\left\{\begin{bmatrix}
		1 & x_{12} & -x_{12}x_{23}\\
		0 & 0 & x_{23}\\
		0 & 0 & 1
		\end{bmatrix},\begin{bmatrix}
		0 & -x_{12} & x_{12}x_{23}\\
		0 & 1 & -x_{23}\\
		0 & 0 & 0
		\end{bmatrix},\begin{bmatrix}
		0 & 0 & y_0\\
		0 & 0 & 0\\
		0 & 0 & 0
		\end{bmatrix}\right\}=\mc{B}_{st},$$
		where $s\coloneqq x_{12}$ and $t\coloneqq x_{23}$.
	\end{proof}
	\smallskip
	
	\begin{thm}\label{no algebra similar to B is projection compressible}
	
		For any $s,t\in\cc$, the algebra $\mc{B}_{st}$ as in Lemma~\ref{unitary orbit for bad algebra B} is not projection compressible. Consequently, no algebra similar to $\mc{B}$ is projection compressible.
	
	\end{thm}
	
	\begin{proof}
	
	Consider the elements $A$ and $B$ of $\mc{B}_{st}$ given by 
	$$\begin{array}{ccc}
	A=\begin{bmatrix}
	1 & s & 0\\
	0 & 0 & t\\
	0 & 0 & 1
	\end{bmatrix} & \text{and} & B=\begin{bmatrix}
	0 & 0 & 1\\
	0 & 0 & 0\\
	0 & 0 & 0
	\end{bmatrix}.	
	\end{array}$$
	We will construct a matrix $P$ that is a multiple of a projection in $\mm_3$, and such that $(PAP)(PBP)$ does not belong to $P\mc{B}_{st} P$. To accomplish this goal, let $k$ be any element of $\rr\setminus\{0,s,t\},$ and define
		$$P\coloneqq \begin{bmatrix}
		k^2+1 & -k & -1\\
		-k & \phantom{-}2 & -k\\
		-1 & -k & k^2+1
		\end{bmatrix}.$$
		Note that $\frac{1}{k^2+2}P$ is a projection in $\mm_3$. 
		
		If $(PAP)(PBP)$ were an element of $P\mc{B}_{st} P$, there would exist a matrix $$C=\begin{bmatrix}	
		\alpha_0 & s(\alpha_0-\beta_0) & x_0\\
		0 & \beta_0 & t(\alpha_0-\beta_0)\\
		0 & 0 & \alpha_0
		\end{bmatrix}\in\mc{B}_{st}$$
		such that $PAPBP-PCP=(g_{ij})$ is equal to $0$. By examining the value of $g_{31}$, one can show that $$x_0=k(\alpha_0-\beta_0+1)(2k-s-t)+2(\alpha_0+1)+k^2\beta_0.\smallskip$$
		From here, direct computations reveal that
		$$(k-s)g_{11}-(k-t)g_{33}=k(k^2+2)(k-s)(k-t).\smallskip$$
		Since $g_{11}=g_{33}=0$, but the right-hand side of the above equation is non-zero by construction, we have reached a contradiction. Thus, there does not exist a $C$ as above, so $P\mc{B}_{st} P$ is not an algebra. The final claim is now a consequence of Lemma~\ref{unitary orbit for bad algebra B}.
	\end{proof}
	\smallskip
	
	\begin{lem}\label{unitary orbit for bad algebra Cr}
	
		Let $\mc{A}$ be a subalgebra of $\mm_3$. If $\mc{A}$ is similar to 
		$$\mc{C}\coloneqq \left\{\begin{bmatrix}
		\alpha & x & y\\ 0& \alpha & x\\ 0 & 0 & \alpha
		\end{bmatrix}:\alpha,x,y\in\cc\right\},$$
		then there is a non-zero constant $r\in\cc$ such that $\mc{A}$ is unitarily equivalent to $$\mc{C}_r\coloneqq \left\{\begin{bmatrix}
		\alpha & x & y\\ 0& \alpha & rx\\ 0 & 0 & \alpha
		\end{bmatrix}:\alpha,x,y\in\cc\right\}.$$
	
	\end{lem}
	
	\begin{proof}
	
		Observe that $\mc{C}$ is spanned by $\left\{I,N_1,N_2\right\}$, where 
		$$\begin{array}{ccc}
		N_1=\begin{bmatrix}
		0 & 0 & 1\\
		0 & 0 & 0\\
		0 & 0 & 0
		\end{bmatrix} & \text{and} & N_2=\begin{bmatrix}
		0 & 1 & 0\\
		0 & 0 & 1\\
		0 & 0 & 0
		\end{bmatrix}
		\end{array}.$$
		Thus, if $S\in\mm_3$ is an invertible matrix such that $\mc{A}=S^{-1}\mc{C} S$, then $\mc{A}$ is spanned by $\left\{I,N_1^\prime, N_2^\prime\right\},$ where $N_i^\prime=S^{-1}N_iS$ for $i\in\{1,2\}$.
		
		 It is evident that $N_1^\prime$ and $N_2^\prime$ are nilpotent operators of rank $1$ and $2$, respectively, and $N_1^\prime N_2^\prime =N_2^\prime N_1^\prime =0$. In particular, since $N_1^\prime$ and $N_2^\prime$ commute, there is a unitary $U\in\mm_3$ such that $U^*N_1^\prime U$ and $U^*N_2^\prime U$ are upper triangular. If $a_{ij}$ and $b_{ij}$ are such that 
		$$\begin{array}{ccc}
		U^*N_1^\prime U=\begin{bmatrix}
		0 & a_{12} & a_{13}\\
		0 & 0 & a_{23}\\
		0 & 0 & 0
		\end{bmatrix} & \text{and} & U^*N_2^\prime U=\begin{bmatrix}
		0 & b_{12} & b_{13}\\
		0 & 0 & b_{23}\\
		0 & 0 & 0
		\end{bmatrix},
		\end{array}$$
		then rank considerations imply that neither $b_{12}$ nor $b_{23}$ is equal to $0$.
		But $$\begin{array}{ccc}
		(U^*N_1^\prime U)(U^*N_2^\prime U)=\begin{bmatrix}
		0 & 0 & a_{12}b_{23}\\
		0 & 0 & 0\\
		0 & 0 & 0
		\end{bmatrix} & \text{and} & (U^*N_2^\prime U)(U^*N_1^\prime U)=\begin{bmatrix}
		0 & 0 & a_{23}b_{12}\\
		0 & 0 & 0\\
		0 & 0 & 0
		\end{bmatrix},
		\end{array}$$
		so it must be that $a_{12}=a_{23}=0$. By setting $r=b_{23}/b_{12}$, it follows that 
		$$U^*\mc{A}U=\mathrm{span}\left\{I,\begin{bmatrix}
		0 & 0 & 1\\
		0 & 0 & 0\\
		0 & 0 & 0
		\end{bmatrix},\begin{bmatrix}
		0 & 1 & b_{13}/b_{12}\\
		0 & 0 & r\\
		0 & 0 & 0
		\end{bmatrix}\right\}=\mc{C}_r.$$
	\end{proof}

	\begin{thm}\label{no algebra similar to Cr is projection compressible theorem}
	
		For every non-zero $r\in\cc$, the algebra $\mc{C}_r$ as in Lemma~\ref{unitary orbit for bad algebra Cr} is not projection compressible. Consequently, no algebra similar to $\mc{C}$ is projection compressible.	
	
	\end{thm}
	
	\begin{proof}
	
		Consider the elements $A,B\in\mc{C}_r$ given by 
		$$\begin{array}{ccc}
		A=\begin{bmatrix}
		0 & 1 & 0\\
		0 & 0 & r\\
		0 & 0 & 0
		\end{bmatrix} & \text{and} & B=\begin{bmatrix} 0 & 0 & 1\\
		0 & 0 & 0\\
		0 & 0 & 0 \end{bmatrix}.
		\end{array}$$ Furthermore, define the matrix $$P\coloneqq \begin{bmatrix}
	\phantom{-}2 & -1 & -1\\
	-1 & \phantom{-}2 & -1\\
	-1 & -1 & \phantom{-}2
	\end{bmatrix},\smallskip$$
	so $\frac{1}{3}P$ is a projection in $\mm_3$. 
	
	We claim that $(PAP)(PBP)$ does not belong to $P\mc{C}_rP$. For if it did, there would exist an element 
	$$C=\begin{bmatrix}
	\alpha_0 & x_0 & y_0\\
	0 & \alpha_0 & rx_0\\
	0 & 0 & \alpha_0
	\end{bmatrix}\smallskip$$ in $\mc{C}_r$ such that $PAPBP-PCP=(g_{ij})$ is equal to $0$. Direct computations show that $$0=g_{31}=3\alpha_0-(x_0+1)(r+1)-y_0,\smallskip$$ hence 
	$y_0=3\alpha_0-(x_0+1)(r+1).$
From here, further calculations reveal that $g_{21}-rg_{32}=3r.$ Since $g_{21}=g_{32}=0$ but $r\neq 0$, we have reached a contradiction. Thus, there does not exist an element $C\in\mc{C}_r$ as described above. This shows that $(PAP)(PBP)\notin P\mc{C}_rP$, so $\mc{C}_r$ is not projection compressible. The final claim is now immediate from Lemma~\ref{unitary orbit for bad algebra Cr}.
	\end{proof}
	\smallskip
	
	\begin{lem}\label{unitary orbit for bad algebra D}
	
		Let $\mc{A}$ be a subalgebra of $\mm_3$. If $\mc{A}$ is similar to 
		$$\mc{D}=\left\{\begin{bmatrix}
		\alpha & 0 & 0\\
		 0 & \beta & 0\\
		 0 & 0 & \gamma
\end{bmatrix}:\alpha,\beta,\gamma\in\cc\right\},$$
then there are constants $r,s,t\in\cc$ such that $\mc{A}$ is unitarily equivalent to 
$$\mc{D}_{rst}\coloneqq \left\{\begin{bmatrix}
		\alpha & r(\alpha-\beta) & s(\alpha-\gamma)-rt(\gamma-\beta)\\
		 0 & \beta & t(\gamma-\beta)\\
		 0 & 0 & \gamma
\end{bmatrix}:\alpha,\beta,\gamma\in\cc\right\}.$$
	
	\end{lem}
	
	\begin{proof}
	
		If $\mc{D}$ is written with respect to the standard basis $\{e_1,e_2,e_3\}$ for $\cc^3$, then $\mc{D}$ is spanned by $\left\{E_{11},E_{22},E_{33}\right\}$ where $E_{jj}=e_j\otimes e_j^*$.	Let $S$ be an invertible element of $\mm_3$ such that $\mc{A}=S^{-1}\mc{D}S$. Clearly $\mc{A}$ is spanned by $\left\{E_{11}^\prime,E_{22}^\prime,E_{33}^\prime\right\}$ where $E_{jj}^\prime=S^{-1}E_{jj}S$.
		
		Observe that the matrices $E_{jj}^\prime$ commute, so there is a unitary $U\in\mm_3$ such that $U^*E_{jj}^\prime U$ is upper triangular for each $j\in\{1,2,3\}.$ Furthermore, since each $U^*E_{jj}^\prime U$ is an idempotent of rank $1$, and $$(U^*E_{ii}^\prime U)(U^*E_{jj}^\prime U)=\delta_{ij} U^*E_{jj}^\prime U\smallskip$$ for all $i$ and $j$, one may re-index the matrices $E_{jj}^\prime$ if necessary to write
		$$\begin{array}{cccc}
		U^*E_{11}^\prime U=\begin{bmatrix}
		1 & x_{12} & x_{13}\\
		0 & 0 & 0\\
		0 & 0 & 0
		\end{bmatrix}, & U^*E_{22}^\prime U=\begin{bmatrix}
		0 & y_{12} & y_{12}y_{23}\\
		0 & 1 & y_{23}\\
		0 & 0 & 0
		\end{bmatrix}, & \text{and} & U^*E_{33}^\prime U=\begin{bmatrix}
		0 & 0 & z_{13}\\
		0 & 0 & z_{23}\\
		0 & 0 & 1
		\end{bmatrix}
		\end{array}$$
	for some $x_{ij}$, $y_{ij}$, and $z_{ij}$ in $\cc$. The fact that these matrices add to $I$ implies that $$\begin{array}{cccc}
	y_{12}=-x_{12}, & y_{23}=-z_{23}, & \text{and} & z_{13}=-x_{13}-x_{12}z_{23}
	\end{array}.$$
	As a result, $$U^*\mc{A}U=\mathrm{span}\left\{\begin{bmatrix}
	1 & x_{12} & x_{13}\\
	0 & 0 & 0\\
	0 & 0 & 0
	\end{bmatrix},\begin{bmatrix}
	0 & -x_{12} & x_{12}z_{23}\\
	0 & 1 & -z_{23}\\
	0 & 0 & 0 
	\end{bmatrix},\begin{bmatrix}
	0 & 0 & -x_{13}-x_{12}z_{23}\\
	0 & 0 & z_{23}\\
	0 & 0 & 1
	\end{bmatrix}\right\}=\mc{D}_{rst},$$
	where $r\coloneqq x_{12}$, $s\coloneqq x_{13}$, and $t\coloneqq z_{23}$.
	\end{proof}	
	\smallskip
	
	\begin{thm}\label{no algebra similar to D is projection compressible}
	
		For any $r,s,t\in\cc$, the algebra $\mc{D}_{rst}$ as in Lemma~\ref{unitary orbit for bad algebra D} is not projection compressible. Consequently, no algebra similar to $\mc{D}$ is projection compressible.
	
	\end{thm}

	\begin{proof}
	
		Consider the elements $A$ and $B$ of $\mc{D}_{rst}$ given by  
		$$\begin{array}{ccc}
		A=\begin{bmatrix}
		1 & r & s\\
		0 & 0 & 0\\
		0 & 0 & 0
		\end{bmatrix} & \text{and} & B=\begin{bmatrix}
		0 & -r & rt\\
		0 & 1 & -t\\
		0 & 0 & 0
		\end{bmatrix}.
		\end{array}$$	
	We wish to construct a matrix $P$ that is a multiple of a projection in $\mm_3$, and such that $(PAP)(PBP)$ does not belong to $P\mc{D}_{rst}P$. To accomplish this goal, choose elements $k,m\in\rr\setminus\{0\}$ subject to the following constraints:
	$$\left.\begin{array}{rrrrrrr}
	\,\,\,\,\phantom{\text{and}} & tk& & & \neq & 1,\\
	\,\,\,\,\phantom{\text{and}} & & & rm & \neq & 1,\\
	\,\,\,\,\phantom{\text{and}} & sk&+&m & \neq & -r,&\,\,\,\,\text{and}\\
	\,\,\,\,\phantom{\text{and}} & k&-&(rt+s)m & \neq & -t.
	\end{array}\right.\smallskip$$ 
	
	\noindent Of course, such $k$ and $m$ always exist. Using these values, define 
	$$P=\begin{bmatrix}
	k^2+1 & -m & -mk\\
	-m & k^2+m^2 & -k\\
	-mk & -k & m^2+1
	\end{bmatrix}.$$
	It is straightforward to check that $\frac{1}{k^2+m^2+1}P$ is a projection in $\mm_3$.
	
	Suppose to the contrary that $(PAP)(PBP)$ were an element of $P\mc{D}_{rst}P$. In this case, there is a matrix 
	$$C=\begin{bmatrix}
		\alpha_0 & r(\alpha_0-\beta_0) & s(\alpha_0-\gamma_0)-rt(\gamma_0-\beta_0)\\
		 0 & \beta_0 & t(\gamma_0-\beta_0)\\
		 0 & 0 & \gamma_0
\end{bmatrix}\in\mc{D}_{rst}$$
such that $PAPBP-PCP=(g_{ij})$ is equal to $0$. We will obtain a contradiction by examining particular entries $g_{ij}$.

Firstly, one may check that 
$$0=g_{31}-kg_{21}=km(k^2+m^2+1)(tk-1)(\beta_0-\gamma_0).\smallskip$$
By construction, the product on the right-hand side is zero if and only if $\beta_0=\gamma_0$. But if this is the case, then $$0=kg_{23}-g_{33}=\beta_0(k^2+m^2+1),\smallskip$$
and therefore $\beta_0=\gamma_0=0.$ Direct computations then show that 
\begin{align*}
(r(k^2+m^2)-sk-m)g_{21}-(k^2&-skm-rm+1)g_{22}\\
&=km(k^2+m^2+1)(rm-1)(sk+m+r)(k-(rt+s)m+t).
\end{align*}
Since $g_{21}=g_{22}=0$ while the right-hand side of this equation is non-zero by construction, we obtain the required contradiction. 

Thus, $(PAP)(PBP)$ does not belong to $P\mc{D}_{rst}P$, so $\mc{D}_{rst}$ is not projection compressible. The final claim now follows from Lemma~\ref{unitary orbit for bad algebra D}.
	\end{proof}
\smallskip
	
	\section{Conclusion}	
	
	Combining Theorems~\ref{no algebra similar to B is projection compressible}, \ref{no algebra similar to Cr is projection compressible theorem}, and \ref{no algebra similar to D is projection compressible} with Theorem~\ref{summarizing thm on compressibility in M3} and its subsequent corollaries, we obtain the following classification of unital subalgebras of $\mm_3$ that admit one, and hence both of the compression properties.
	
	\begin{thm}\label{main thm for compressibility in M3}
	
		If $\mc{A}$ is a unital subalgebra of $\mm_3$, then the following are equivalent:
		\begin{itemize}
			\item[(i)]$\mc{A}$ is projection compressible;
			
			\item[(ii)]$\mc{A}$ is idempotent compressible;
			
			\item[(iii)] $\mc{A}$ is the unitization of an $\mc{LR}$-algebra, or $\mc{A}$ is not $3$-dimensional;
	
	\item[(iv)] $\mc{A}$ is not singly generated by an invertible nonderogatory matrix;
			
			
			\item[(v)]$\mc{A}$ is the unitization of an $\mc{LR}$-algebra, or $\mc{A}$ is transpose similar to one of the algebras from \linebreak Theorem~\ref{summarizing thm on compressibility in M3}(i).

		\end{itemize}		
	
	\end{thm}

%
%
%
%
%
%
%
%
%
%
	
	The fact that the set of projection compressible and idempotent compressible subalgebras of $\mm_3$ (and as will be shown in \cite{ZCramerCompressibility}, of $\mm_n$ for all $n\geq 4$) coincide is rather surprising.  As mentioned in the introduction, despite a considerable amount of effort, we have been unable to provide a direct proof of this fact that does not involve characterizing each class of algebras. Such a proof might shed further light on why these algebras have the particular structures described above.\\
	
\section*{Acknowledgements} \noindent The authors would like to thank Janez Bernik and Bamdad Yahaghi for many helpful conversations.\\

	\bibliography{CMR1Bib}
	\bibliographystyle{plain}

	\vspace{0.1cm}
		
	\Addresses
	
\end{document}